\documentclass[a4paper]{article}
\usepackage{graphicx}
\usepackage{amsfonts}
\usepackage{amsmath}
\usepackage{amsthm}
\usepackage{url}
\usepackage{subcaption}
\usepackage[left=3cm,top=3.0cm,right=3cm,bottom=3.0cm]{geometry}
\usepackage{authblk}

\usepackage{ngkarris-math}

\usepackage{xcolor}

\newtheorem{theorem}{Theorem}[section]
\newtheorem{remark}{Remark}
\newcommand{\norm}[1]{\left\lVert#1\right\rVert}

\title{Optimal Transport, Timesteppers, Newton-Krylov Methods \\
and Steady States of Collective Particle Dynamics.}

\author[1]{Hannes Vandecasteele\thanks{\texttt{hvandec@jh.edu}}}
\author[2]{Nicholas Karris\thanks{\texttt{nkarris@ucsd.edu}}}
\author[3]{Alexander Cloninger\thanks{\texttt{acloninger@ucsd.edu}}} 
\author[1,4]{Ioannis G. Kevrekidis\thanks{\texttt{yannisk@jhu.edu}}}

\affil[1]{Department of Applied Mathematics and Statistics, Johns Hopkins University}
\affil[2]{Department of Mathematics, University of San Diego, California}
\affil[3]{Department of Mathematics \& Halicioğlu Data Science Institute (HDSI), University of San Diego, California}
\affil[4]{Department of Chemical and Biomolecular Engineering, Johns Hopkins University}

\date{\today}

\begin{document}
\maketitle

\begin{abstract}
Timesteppers constitute a powerful tool in modern computational science and engineering. Although they are typically used to advance the system forward in time, they can also be viewed as nonlinear mappings that implicitly encode steady states and stability information. In this work, we present an extension of the matrix-free framework for calculating, via timesteppers, steady states of deterministic systems  to stochastic particle simulations, where intrinsic randomness prevents direct steady state extraction. By formulating stochastic timesteppers in the language of optimal transport, we reinterpret them as operators acting on probability measures rather than on individual particle trajectories. This perspective enables the construction of smooth cumulative- and inverse-cumulative-distribution-function ((I)CDF) timesteppers that evolve distributions rather than particles. Combined with matrix-free Newton–Krylov solvers, these smooth timesteppers allow efficient computation of steady-state distributions even under high stochastic noise. We perform an error analysis quantifying how noise affects finite-difference Jacobian action approximations, and demonstrate that convergence can be obtained even in high noise regimes. Finally, we introduce higher-dimensional generalizations based on smooth CDF-related representations of particles and validate their performance on a non-trivial two-dimensional distribution. Together, these developments establish a unified variational framework for computing meaningful steady states of both deterministic and stochastic timesteppers.
\end{abstract}

\section{Introduction}
Timesteppers are a powerful tool in modern scientific computation. Given the state $u(t)$ of our system at time $t$, i.e., a function or a collection of particles, and a time horizon $h$, a timestepper $\phi_{h}$ advances the state in time
\begin{equation}
    u(t+h) = \phi_{h}\left(u(t)\right),
\end{equation}
so that iterating the map generates a trajectory $\{u\left(n h\right)\}_{n=0}^\infty$. For stable iteration, the limit of $u\left(nh\right)$ approximates the (stable) steady-state solution of the underlying model. In the context of this manuscript, this model will typically be a differential equation of some kind. Traditionally, timesteppers are used ``naively": we apply $\phi_{h}$ repeatedly to obtain a numerical trajectory that is subsequently studied for qualitative and quantitative convergence features. 

Timesteppers, however, can also be usefully exploited beyond direct simulation. They embody everything about the problem, namely the underlying differential equation, the model parameters, and the boundary conditions. Instead of treating them only as simulators, one may ask what information about the system can be extracted directly from the timestepper mapping itself. Several important system-level tasks can be formulated in this way
\begin{itemize}
\item[1.] Steady States. Fixed points $u^* = \phi_{h}\left(u^*\right)$ of the timestepper correspond to the steady (or invariant) states of the underlying differential equation.
\item[2.] Eigenvalues and Stability. The Jacobian $J = D\phi_{h}(u^*)$ of the timestepper at $u^*$ contains all spectral information from which eigenvalues and stability properties of the underlying model can be inferred. There is a direct relationship between the eigenvalues of the timestepper and the eigenvalues of the right-hand side of the differential equation.
\item[3.] Control and Identification. By systematically tracking the steady-state solutions across varying parameter values, we can trace the corresponding solution branches and identify critical points where the system’s behavior changes drastically, such as folds, bifurcations, or Hopf bifurcations.
\end{itemize}
A convenient way to unify these three tasks is through the residual mapping
\begin{equation} \label{eq:psi}
    \psi(u) = u - \phi_{h}\left(u\right).
\end{equation}
Here, $u \in \mathbb{R}^d$ represents the solution of some differential equation at a given time, typically on a grid, and $\phi_h, \psi: \mathbb{R}^d \to \mathbb{R}^d$ are vector functions. This formulation for calculating steady states allows us to apply the full machinery of iterative nonlinear solvers through matrix-free methods. Particularly in previous work~\cite{fabiani2024task,qiao2006spatially,kelley2004newton}, we used the Newton--Krylov and Arnoldi methods to calculate steady-state profiles and the leading eigenvalues of the Jacobian, respectively. These methods are matrix-free. We do not need explicit Jacobian matrices to compute the steady state or the leading eigenvalues. This last part is essential because codes that implement a timestepper rarely have their Jacobians coded as well. Even if they do so in the form of hand-coded formulas or automatic differentiation, evaluating the (typically large) Jacobian is too time-consuming and uses a lot of memory. Matrix-free methods avoid these bottlenecks by approximating the action of the Jacobian $D \psi(u) v$ in a particular direction $v$ on the fly using finite differences
\begin{equation} \label{eq:JVP}
    D \psi(u) v \approx \frac{\psi(u + \varepsilon v) - \psi(u)}{\varepsilon}.
\end{equation}
Although this approximation induces an error, it can often be ignored in view of the speedups that can be obtained.

In our previous work~\cite{kelley2004newton,qiao2006spatially}, we focused on deterministic timesteppers, developing Newton--Krylov and other matrix-free methods to compute steady states and stability information directly from the residual $\psi(u)$. Most recently we also applied the Newton--Krylov method to neural operators~\cite{fabiani2025enabling}. This line of research established how timesteppers could serve as fixed-point solvers for macroscopic equations, particularly in gradient systems where convergence is theoretically guaranteed. 

In stochastic settings, the naive formulation $X - \phi_{h}(X) = 0$ for a steady state particle ensemble $X$ is ill-defined. Timesteppers are no longer deterministic mappings, but rather involve random evolutions of particle ensembles. Individual particles do not carry steady-state information directly. Indeed, even in steady state, the particles typically fluctuate because of thermal noise. This does not mean that there is no steady state distribution of stochastic particles, but only that the mathematical equality $X = \phi_h(X)$ does not hold. As a consequence, solving $\psi(X) = 0$ for a particle ensemble using Newton-Krylov is useless. In this paper, we re-interpret what a steady-state of particles means on the level of this timestepper residual.

The idea of computing steady-state distributions from particle timesteppers is not new. In~\cite{siettos2012equation}, the authors average many independent particle simulations to obtain smooth, macroscopic quantities. They propose a smooth coarse timestepper in terms of these macroscopic variables and use the Newton–Krylov method to compute (macroscopic) steady states and bifurcation diagrams. A similar idea appears in~\cite{qiao2006spatially}, where an approximate coarse model is introduced to precondition the Newton–Krylov method. The authors also used this preconditioned matrix-free framework to calculate the leading eigenvalues of the Jacobian and to track steady states along the bifurcation diagram. For self-similar systems, coarse representations based on cumulative distribution functions (CDFs) have also been used; see~\cite{chen2004molecular,zou2005equation,zou2006equation}. In these works, the coarse CDF description is repeatedly lifted to fine-scale particle ensembles --- typically via inverse CDF (ICDF)–based sampling --- in order to evolve the system forward in time and ultimately compute steady or self-similar states. Finally, we also mention the work of~\cite{siettos2010system} where kernel-density estimation was applied to stochastic timesteppers in bacterial chemotaxis to calculate the steady-state distribution.

In this manuscript, we go back to the original formulation~\eqref{eq:psi} of steady states to understand what can be recovered when the timestepper is stochastic. As we said earlier, a steady-state ensemble $X$ does not solve $\psi(X) = 0$, but some information must still be contained in this residual. Optimal Transport (OT) provides a principled way to compare ensembles in a Wasserstein space, enabling the construction of well-defined particle-based timesteppers. This extension allows us to treat both deterministic and stochastic dynamics within a unified variational framework. From an optimal transport perspective, we treat the evolving distribution of particles as a time dependent probability measure $\mu_t$ on $\mathbb{R}^d$, rather than trying to follow each particle individually. If $\mu_t$ is ``smooth'' in the sense of having a density and evolving in a regular way, then there is a velocity field $v_t$ such that the continuity equation
\begin{equation} \label{eq:OTcontinuity}
\partial_t \mu_t + \nabla \cdot (v_t \mu_t) = 0
\end{equation}
describes how mass flows. The minimal energy (and smoothest) velocity field $v_t$ can be computed via the optimal transport map $T_{\mu_t}^{\mu_{t+h}}$:
\[v_t = \lim_{h\to 0}\frac{T_{\mu_t}^{\mu_{t+h}} - \id}{h}.\]
We have previously shown that under certain assumptions on the smoothness of $(\mu_t)_{t\ge 0}$, forward Euler methods using $v_t$ are first order accurate in the Wasserstein distance $W_2$ \cite{karris2024using}.

The key advantage over tracking particles individually (which in stochastic or chaotic systems can be very noisy) is that this optimal transport--based method aggregates the motion of \emph{mass} rather than individual sample paths. In many systems, especially those with fast/slow scales, individual particle trajectories will jitter, be non-differentiable, or otherwise provide a noisy or misleading picture at short time scales. By contrast, the distribution $\mu_t$ often changes in a much smoother and more coherent manner. By fitting the velocity field from optimal transport between successive empirical distributions, one effectively filters out the microscopic stochasticity and focuses on how the bulk of the particles move. This produces a cleaner signal, which will allow for a more accurate Newton--Krylov step.
In essence, optimal transport allows us to take a ``Lagrangian'' perspective without needing to actually track each individual particle.
The velocity fields we get from OT are smoother and better describe the overall distribution's behavior, but they still fundamentally capture the dynamics from the perspective of individual agents (as opposed to looking at a particular location, as in the ``Eulerian'' perspective).

These optimal transport algorithms are especially efficient in one dimension. In fact, the optimal transport map of a set of particles can be obtained simply by sorting the particles. 
This is because the optimal transport map from the uniform distribution on \([0,1]\) to any probability distribution \(\mu\) is exactly the inverse of the cumulative density function (ICDF) of \(\mu\) ~\cite[Chapter 2]{santambrogio2015optimal}. 
Moreover, in one dimension, we can compose optimal transport maps to produce new optimal transport maps (i.e., \(T_{\mu_t}^{\mu_{t+2h}} = T_{\mu_{t+h}}^{\mu_{t+2h}} \circ T_{\mu_t}^{\mu_{t+h}}\)), and so computing optimal transport maps is nothing more than computing CDFs and ICDFs.
The advantage of the CDF/ICDF formulation is that inverse CDFs are always smooth (as long as the underlying distribution is absolutely continuous), which allows us to employ advanced optimization techniques for calculating steady states of particle systems. Indeed, a smooth macroscopic ICDF-to-ICDF timestepper can be built directly in three steps: 1) Sample the ICDF by evaluating it in percentiles; 2) Propagate the particles using the stochastic timestepper; 3) Compute the new ICDF from the particle locations. This scheme gives us a bridge between microscopic particle simulations and macroscopic evolution, and is reminiscent of equation-free methods~\cite{kevrekidis2002equation,kevrekidis2004equation,geary2001projective}.

In multiple dimensions, while the optimal transport velocity fields are still the appropriate objects to consider, computing them is much more challenging.
Instead of simply sorting the particles (or computing a CDF), one needs to solve a linear program, which is much more computationally expensive.
Thus, in practice, one must construct an approximation in a systematic way.
For example, while the inverse cumulative density function is not well-defined in higher dimensions, the regular cumulative density function does carry over to multiple dimensions.
This means we can construct a CDF-to-CDF timestepper using the same three steps: 1) Generate samples from the multidimensional CDF; 2) Propagate the particles using the stochastic timestepper; 3) Compute the new CDF from the particles. However, the first step becomes nontrivial in multiple dimensions because simply inverting the CDF at `percentiles' for sampling is impossible. We need to project the multidimensional CDF down to its one-dimensional representations to generate samples. One such technique is to use marginal and conditional CDFs~\cite{zou2005equation,zou2006equation}. First, we construct the marginal CDF in the first dimension and generate $x$-samples from this 1D CDF by inverting it. Then, for every 1D sample, we construct the conditional CDF of the second dimension $y$ and invert it. This procedure is analogous to the OT/ICDF method described above, except it comes from using the Knothe-Rosenblatt coupling instead of the true optimal coupling.
Knothe-Rosenblatt maps can be expressed as a limit of optimal transport maps with appropriate cost functions, and so we can interpret them as approximations of OT maps.
This Knothe--Rosenblatt formulation allows for an efficient extension of an OT-based timestepper in higher dimensions, though at the cost of introducing asymmetry. Other approaches based on sliced optimal transport (i.e., the sliced Wasserstein distance) also offer a flexible alternative, projecting the problem onto many one-dimensional slices before reassembling the results~\cite{bonneel2015sliced,kolouri2019generalized}. Both types of sampling strategies are crucial to making OT-based timesteppers practical for complex and higher-dimensional systems.

Having reinterpreted stochastic timesteppers in optimal transport terms, we can now deploy advanced numerical solvers to compute steady-state distributions. However, there is one computational roadblock. The (I)CDF-to-(I)CDF timestepper is itself stochastic because it is built around a particle timestepper. Averaging over independent runs or variance reduction might reduce this stochastic variability~\cite{qiao2006spatially}, but evaluations of $\psi$ will always be non-deterministic. Indeed, if we evaluated $\psi$ in $u + \varepsilon v$ and $u$ independently, with the respective stochastic error terms $\xi_1$ and $\xi_2$, we effectively calculated $\tilde{\psi}(u) = \psi(u) + \xi_1$ and $\tilde{\psi}(u + \varepsilon v) = \psi(u+ \varepsilon v) + \xi_2$. The Jacobian-vector product in equation~\eqref{eq:JVP} then becomes
\begin{equation} \frac{\tilde{\psi}(u+\varepsilon v) - \tilde{\psi}(u)}{\varepsilon} = \frac{\psi(u+\varepsilon v) + \xi_2 - \psi(u) - \xi_1}{\varepsilon} = \frac{\psi(u+\varepsilon v) - \psi(u)}{\varepsilon} + \frac{\xi_2 - \xi_1}{\varepsilon}.
\end{equation}
At first sight, the error in the finite-difference approximation of the Jacobian looks catastrophic. Indeed, since $\xi_1$ and $\xi_2$ are independent (and we can assume the same error distribution), the variance of $\xi_2-\xi_1$ is double that of $\xi_1$. We further divide it by a (typically) small step size $\varepsilon$! To keep the stochastic error within reasonable limits, we must use a large $\varepsilon$, which in turn induces a larger approximation error in~\eqref{eq:JVP}. In this paper, we will connect the existing error analysis of the (stochastic-free) Newton--Krylov method, and particularly of the Jacobian-vector products, with this stochastic error. We will demonstrate, both theoretically and numerically, that one can still compute steady-state (I)CDFs of stochastic systems accurately, up to a certain noise level or tolerance. We will also demonstrate numerically that the `optimal' finite difference step size $\varepsilon$ increases from $10^{-8}$ to $10^{-2}$ -- $10^{-1}$ in double precision.

\paragraph{Contributions of this work}
This work extends the timestepper-based framework for steady-state computation from deterministic to stochastic particle systems by combining matrix-free Newton–Krylov methods with Optimal Transport formulations. Our main contributions are as follows:
\begin{itemize}
    \item[1.] \textbf{Generalization of timestepper-based fixed-point computation to stochastic systems.} We extend the classical residual formulation $\psi(u) = u - \phi_h(u)$ to stochastic particle timesteppers, where individual particle mappings are ill-defined due to randomness.
    \item[2.] \textbf{Formulation of stochastic timesteppers in the language of Optimal Transport.} We reformulate timesteppers as operations acting on probability measures, and use the Wasserstein distance to define a well-posed residual that is minimized at steady state.
    \item[3.] \textbf{Development of smooth distribution-based timesteppers.} We develop smooth (I)CDF-to-(I)CDF timesteppers that bridge stochastic simulations with distributional evolution, enabling the use of higher-order optimization such as the Newton--Krylov method.
    \item[4.] \textbf{Error analysis of Newton–Krylov methods in the stochastic setting.}  We extend the theoretical error analysis of the Newton--Krylov method in~\cite{brown2008using} to include stochastic timesteppers, providing a quantitative trade-off between directional derivative step size and the number of stochastic particles.
    \item[5.] \textbf{Efficient extensions to higher dimensions.} We propose two practical multidimensional generalizations: a CDF-to-CDF timestepper based on marginal and conditional CDFs, and a sliced Wasserstein formulation relying on (random) one-dimensional projections.
    \item[6.] \textbf{Numerical validation and demonstrations.} We demonstrate the performance of the proposed methods on representative stochastic systems, including a two-dimensional distribution, highlighting accuracy and robustness of Newton--Krylov on smooth representations of particle distributions.
\end{itemize}

\paragraph{Outline of the Manuscript}
We start this paper in Section~\ref{sec:w2_stochastics} by discussing steady-state distributions of particle systems in the language of optimal transport. We particularly focus on the notion of steady state through a reformulation of~\eqref{eq:psi} using the Wasserstein distance. We further relate these as steady-states of the Euler-OT timestepper and discuss its connection to the OT map and the velocity field. From these fundamentals, we build a first-order optimization scheme by minimizing the Wasserstein distance between particles and a time-evolved copy of those particles. We particularly derive an Adam--Wasserstein optimizer and demonstrate, theoretically and numerically on three examples, that this Wasserstein optimization scheme converges quickly and smoothly to the correct steady-state particle distribution. Next, in Section~\ref{sec:nk} we introduce the matrix-free Newton--Krylov method. We perform a general error analysis in the non-stochastic case and demonstrate convergence of Newton--Krylov on the mean-field equation associated with particle methods. We also discuss how the integration time horizon $h$ can be optimally chosen through a spectral analysis of the timestepper.  We also perform an extended error analysis of the Newton--Krylov method in the stochastic case (see Section~\ref{subsec:nk_noisy}) and explain how second-order optimizers applied to this Wasserstein formulation suffer from large statistical errors. This result motivates us to investigate smooth representations of particle distributions, and in particular the ICDF-to-ICDF timestepper in Section~\ref{sec:particles_to_smooth}. We explain the main idea in detail and demonstrate an improved convergence of the Newton--Krylov method using these smooth representations. Finally, in Section~\ref{sec:higher_dimensions}, we discuss two alternative smooth representations for higher-dimensional systems: the CDF-to-CDF timestepper and its equivalent sliced Wasserstein formulation. We then demonstrate their effectiveness on a two-dimensional half-moon probability distribution. We conclude this paper with a summarizing discussion and outlook for further research in Section~\ref{sec:conclusion}.

\section{Stochastic Particle Steady States and Optimal Transport} \label{sec:w2_stochastics}

In previous work, we developed a Newton--Krylov method for computing steady states of deterministic timesteppers.
In the deterministic setting, we wrote the fixed-point condition in residual form
\begin{equation}
\psi(x) = x-\phi_h(x) = 0.
\end{equation}
For particle timesteppers, however, a one-step map is inherently random. That is, $y=\phi_h(x;\xi)$ where $\xi$ is the collection of random numbers used in propagating the particles.
Even when the system is ``in equilibrium'', the individual particles keep moving due to the intrinsic noise $\xi$. An example is the random motion of molecules, even in thermal and statistical equilibrium. More generally, particles move at certain probabilities and `detailed balance' must be maintained. Hence the point-wise condition $\psi(x)=x-\phi_h(x)=0$ is not well defined.

Instead, the steady state must be understood at the level of distributions: a distribution $\mu^\star$ is stationary when the distribution obtained after one step from $\mu^\star$ coincides with $\mu^\star$ again. However, distributions are not readily available from particle codes and must be built using histograms or tools like kernel density estimation. These methods for approximating the underlying distribution are sometimes as noisy as the particles themselves, and they are resource-intense.

To address this problem, we cast this distributional steady state as an optimization problem. We work in Wasserstein geometry and measure the mismatch between the current distribution and its one-step image using the $2$-Wasserstein distance. Writing $X \sim \mu$ and $Y=\phi_h(X,\xi)$, the objective
\begin{equation}
    \mathcal{J}(\mu) = \tfrac12\,W_2^2\!\big(\mu,\ \text{law}(Y)\big)
\end{equation}
is nonnegative and vanishes if and only if $\mu$ is stationary. The Wasserstein distance directly ties into the heart of Optimal Transport theory, and we present some relevant results here.

\subsection{Optimal Transport Background}

Much of the background we present here is discussed in more detail in our previous work \cite{karris2024using}.
We highlight the major results that will be particularly important for us in this paper, but for a more complete presentation of the theory, we refer the reader to \cite{karris2024using}.

Let \(\cP_2(\R^d)\) be the set of Borel probability measures on \(\R^d\) with finite second moment, i.e., the probability measures \(\mu\) such that \(\int_{\R^d}\|x\|_2^2\,d\mu(x) < \infty\).
For \(\sigma, \mu\in\cP_2(\R^d)\), we define \(\Gamma_{\sigma, \mu}\) to be the set of all couplings between \(\sigma\) and \(\mu\):
\begin{equation}
    \Gamma_{\sigma,\mu} \defeq \{\gamma\in \cP_2(\R^d\times\R^d) : \gamma(A\times \R^d) = \sigma(A) \text{ and } \gamma(\R^d\times A) = \mu(A) \text{ for all Borel } A\subseteq \R^d\}. 
\end{equation}
We then define the (2-)Wasserstein distance between \(\sigma\) and \(\mu\) to be
\begin{equation}\label{eq:wass-dist-defn}
    W_2(\sigma,\mu) \defeq \min_{\gamma\in \Gamma_{\sigma,\mu}}\left(\int_{\R^d\times\R^d} \|x-y\|^2\,d\gamma(x,y)\right)^{\frac12}.
\end{equation}
The coupling \(\gamma\) that satisfies \eqref{eq:wass-dist-defn} is called an ``optimal coupling'' or an ``optimal plan''.

It is well known that \(W_2\) defines a metric on \(\cP_2(\R^d)\) (see, e.g., \cite{peyre2019computational}).
It is also well known that the optimal coupling \eqref{eq:wass-dist-defn} can be satisfied by a proper map $x \mapsto y = T(x)$ under hypotheses:
\begin{theorem}[Brenier \cite{brenier1991polar}] \label{thm:brenier}
    Let \(\sigma,\mu\in\cP_2(\R^d)\).
    If \(\sigma\) has a density with respect to the Lebesgue measure, then the optimal coupling \(\gamma\) that satisfies \eqref{eq:wass-dist-defn} is unique, and there exists a \(\sigma\)-a.e.\ unique map \(T:\R^d\to\R^d\) such that \(\gamma = (\id, T)_\#\sigma\).
    That is,
    \begin{equation}\label{eq:brenier-wass-dist}
        W_2(\sigma,\mu) = \left(\int_{\R^d}\|x-T(x)\|_2^2\,d\sigma(x)\right)^{\frac12}.
    \end{equation}
    Moreover, there exists a \(\sigma\)-a.e.\ unique (up to additive constant) convex function \(\phi\) such that \(T = \nabla\phi\).
\end{theorem}
We call the map \(T\) given in Theorem \ref{thm:brenier} the ``optimal transport map'' from \(\sigma\) to \(\mu\), and we denote it \(T_\sigma^\mu\).

While optimal transport on discrete measures does not generally admit optimal \emph{maps} as in Brenier's Theorem, there is an important case where we do get optimal maps:

\begin{theorem} [Proposition 2.1 in \cite{peyre2019computational}]
    If \(\sigma = \frac1N\sum_{i=1}^N \delta_{x_i}\) and \(\mu = \frac1N\sum_{i=1}^N \delta_{y_i}\) are uniform measures on the same number of distinct points, then there exists an optimal transport map \(T_\sigma^\mu\) in the sense that
    \[
        W_2(\sigma,\mu) = \left(\frac1N\sum_{i=1}^N\|x_i - T_\sigma^\mu(x_i)\|^2\right)^{\frac12}
    \]
    is minimal. Moreover, if \(T_\sigma^\mu\) is such a map, then there exists a permutation \(\tau\in S_N\) (the symmetric group on \(N\) elements) such that \(T_\sigma^\mu(x_i) = y_{\tau(i)}\) for all \(i = 1, \dots, N\).
\end{theorem}

\subsection{Evolving Measures as Curves in Wasserstein Space}
Now we consider the setting where we have a probability measure evolving over time.
Explicitly, consider a curve \(\mu:[0,T]\to \cP_2(\R^d)\) written as \(t\mapsto \mu_t\).
If \(\mu_t\) is absolutely continuous (as a curve) with respect to the Wasserstein distance, then there exists a unique ``minimal energy'' velocity field \(\bv_t\) that satisfies the continuity equation
\[
    \del_t\mu_t + \nabla\cdot(\bv_t\mu_t) = 0,
\]
and if \(\mu_t\) has a density with respect to the Lebesgue measure for all times \(t\in[0,T]\), then this velocity field is computable as a derivative of optimal transport maps:
\begin{equation}\label{eq:velocity-field-limit-computation}
    \bv_t = \lim_{h\to 0} \frac{T_{\mu_t}^{\mu_{t+h}}-\id}{h} \qquad \text{ in } L^2(\mu_t).
\end{equation}
Furthermore, the Benamou-Brenier formulation of the Wasserstein distance \cite{benamou2000computational} gives
\[
    W_2(\mu_t,\mu_{t+h}) = \left(\int_t^{t+h} \|\bv_s\|_{L^2(\mu_s)}^2\,ds\right)^{\frac12}.
\]
Notice that the above formulation says that an evolving distribution reaches steady state exactly when \(\bv_t = 0\).
Moreover, when \(\bv_s \approx 0\) (meaning \(\|\bv_s\|_{L^2(\mu_s)}\approx 0\)) for all \(s\in[t,t+h]\), then \(W_2(\mu_t,\mu_{t+h}) \approx 0\), and so the distribution has ``approximately'' reached steady state.

In our setting, however, we do not have access to \(\mu_t\) for a continuous time interval.
Instead, we have a timestepper \(\phi_h\) for some small step \(h\), and so we only have \(\mu_t\) for discrete times \(t_n = hn\).
In this situation, we assume that the timestepper approximates the evolution of some continuous curve, and that we can approximate the velocity field \(\bv_t\) associated with this continuous curve using the finite-difference approximation suggested by \eqref{eq:velocity-field-limit-computation}:
\[
    \bv_t \approx \bv_t^h \defeq \frac{T_{\mu_t}^{\mu_{t+h}}-\id}{h}.
\]
Note that by \eqref{eq:brenier-wass-dist}, we have \(W_2(\mu_t,\mu_{t+h}) = h\|\bv_t^h\|_{L^2(\mu_t)}\), and so computing the Wasserstein distance between successive time steps is equivalent to computing the norm of the approximate velocity field.
For more details about the link between evolving measures, velocity flow fields, and finite difference approximations, we refer to our previous work in \cite{karris2024using}.

\subsection{Wasserstein Formulation of Particle Steady States} \label{subsec:wasserstein_optimization}


In practice, our timestepper acts on particles, not on continuous distributions. We can still use the ideas discussed above as a heuristic to suggest that the system's evolution is in (approximate) steady state when the Wasserstein distance between successive time steps is minimized.

That is, if we consider the initial locations of the particles \(X = (x_1, \dots, x_N)\) and their locations \(Y = \phi_h(X) = (y_1, \dots, y_N)\) - after applying the timestepper over a time-horizon of size $h$ - as discrete probability measures \(\mu = \frac1N\sum_{i=1}^N \delta_{x_i}\) and \(\nu = \frac1N\sum_{i=1}^N \delta_{y_i}\) respectively, then we can consider the objective function
\begin{equation} \label{eq:discrete_W2}
    F(X;\xi) \defeq \frac{N}{2} W_2^2(\mu,\nu) = \frac{1}{2}\sum_{i=1}^N\|x_i - T_\mu^\nu(x_i)\|^2 = \frac{1}{2}\sum_{i=1}^N\|x_i - y_{\sigma^\star(i)}\|^2,
\end{equation}
where \(\sigma^\star\) is the optimal coupling - a permutation. Finding an approximate steady state is equivalent to (approximately) finding a minimizer of \(F\) over all initial particle locations \(X\).
Between changes of the optimal transport coupling, $F$ is a smooth quadratic in the particle locations, and its (sub)gradient with respect to each $x_i$ is the transport displacement from $x_i$ to its match under the current optimal plan. This naturally leads us to first-order optimization methods to minimize~\eqref{eq:discrete_W2} such as (stochastic) gradient descent or Adam~\cite{kingma2014adam}. 

Let \(\Pi^\star\) denote the permutation matrix that corresponds to the optimal coupling, i.e., \(\Pi^\star Y = (y_{\sigma^\star(1)}, \dots, y_{\sigma^\star(N)})\).
Now \(F\) reads
\[
    F(X;\xi) = \frac12\|X - \Pi^\star Y\|^2
\]
and so the gradient of $F$ with respect to $X$ is
\begin{equation} \label{eq:w2_gradient}
    \nabla F(X) = X - \Pi^\star Y + D\phi_h(X; \xi)^T\left(\phi_h(X,\xi) - (\Pi^\star)^{\top} X\right).
\end{equation}
In principle, we can evaluate this gradient and plug it into any first-order optimization algorithm, such as the Adam optimizer.

From a numerical point of view, the second term of~\eqref{eq:w2_gradient} poses a problem. Two subsequent evaluations of $F(X)$ for the same particles $X$ will differ due to the influence of $\xi$. In particular, this means that the Jacobian of the timestepper $D\phi_h(X; \xi)$ will be very noisy, even if we can evaluate it exactly by automatic differentiation. More importantly, this Jacobian will contain little actual information about the drift of the particles. Fortunately, by the Envelope theorem~\cite{ambrosio2005gradient,peyre2019computational}, we can use only the first term of the gradient
\begin{equation}
\partial_X \tfrac{1}{2}W_2^2(X,Y) = X - \Pi^\star Y
\end{equation}
for optimization purposes.

\begin{remark}
    This formulation gives a slightly different perspective on the Euler-type method we develop in detail in our previous work \cite{karris2024using}. 
    Note that, in the language of velocity fields from earlier, we have
    \[
        \bv_t^h(x_i) = \frac{T_\mu^\nu(x_i) - x_i}{h} = \frac{1}{h}(y_{\sigma^\star(i)}-x_i),
    \]
    and so this partial gradient of our objective function is nothing more than (a scaled copy of) the negative of our approximate velocity field \(\bv_t^h\).
    This means that our Euler-type method of taking steps in the direction described by the velocity field \(\bv_t^h\) can actually be interpreted as running gradient descent with the partial gradient \(X-\Pi^\star Y\).
    The insight that the present functional-gradient formulation provides, is that it is now clear how to apply more sophisticated methods for minimizing the objective function (i.e., finding steady states), such as Adam, which we can interpret as a natural extension of our previous work.
\end{remark}

There remains one more computational bottleneck in this Wasserstein-minimization algorithm, and that is the computation of the optimal transport map $\Pi^\star$. In one dimension, sorting both clouds and matching in order yields $\Pi^\star$ in $\mathcal{O}(N\log N)$ time. This is the theoretically optimal complexity. In multiple dimensions, one may use a linear assignment solver which is typically very expensive, i.e., $\mathcal{O}(N^3)$~\cite{burkard2012assignment,peyre2019computational}. We will showcase the Wasserstein-minimization approach on both one- and two-dimensional examples in the next section.

For any initial set of particles $X$, we can evaluate $F(X)$ and its gradient $\partial_X F(X)$. We can therefore minimize $F$ with any first-order scheme; we use Adam because of its stability under mild noise. A typical iteration is as follows
\begin{enumerate}
\item Given $X_k$ and random numbers $\xi_k$, compute $Y_k=\phi_h(X_k;\xi_k)$.
\item Compute an optimal plan $\Pi_k^\star=\Pi^\star(X_k,Y_k)$.
\item Form the gradient $g_k=X_k-\Pi_k^\star Y_k$ (optionally average over a small batch of independent runs for variance reduction).
\item Apply Adam to obtain $X_{k+1}$.
\end{enumerate}
We use PyTorch's Adam optimizer in our implementation. One advantage of PyTorch is that we can use its call graph to directly compute gradients of $F(X)$. The first term of that gradient can then be obtained by detaching $Y$ from the call graph.

In summary, minimizing $F(X)=\tfrac12 W_2^2(X,\phi_h(X;\xi))$ with Adam yields a fast $\mathcal{O}(N \log N)$ procedure that is robust to stochasticity and faithful to the Wasserstein fixed-point formulation, while avoiding the instability of second-order information at coupling changes.

\subsection{Three Examples} \label{subsec:wasserstein_results}
The Wasserstein particle flow is a well-defined, easy-to-implement, and reliable approach to compute steady-state distributions of particle methods. Combining gradient flow with the Adam optimizer also has significant advantages such as steady decrease of the $W_2^2$-loss. In this section, we demonstrate the combined method on three examples: bacterial chemotaxis (section~\ref{subsubsec:w2_chemotaxis}), the nonlinear economic agents model (section~\ref{subsubsec:w2_agents}), and a two-dimensional overdamped Langevin dynamics with non-trivial potential energy profile (section~\ref{subsubsec:w2_halfmoon}).

\subsubsection{Bacterial Chemotaxis} \label{subsubsec:w2_chemotaxis}
Consider a collection of bacteria $X_t$ that eat food from a substrate $S(x)$ on the domain $[-L,L]$. The function $S(x)$ represents the amount of food at position $x$, and we assume that $S$ remains unchanged over time. Furthermore, there is a chemotactic sensitivity $\chi(S)$ that indicates the level of activation of the bacteria, solely a function of the local food supply. Finally, we also incorporate a diffusion constant $D$ to model the random motion. The bacterial positions $X_t$ over time are then governed by the overdamped Langevin equation~\cite{fournier2017stochastic}
\begin{equation} \label{eq:chemotaxis}
 dX_t = \chi(S(X_t)) S_x(X_t)dt + \sqrt{2D} dW_t,
\end{equation}
where $S_x$ is the food gradient and $W_t$ is the standard Brownian motion. The bacteria also cannot leave $[-L,L]$ so we impose reflective or no-flux (Neumann) boundary conditions.

We discretize the overdamped Langevin dynamics in equation~\eqref{eq:chemotaxis} using an Euler-Maruyama method with step size $10^{-3}$, and per evaluation of the timestepper $\phi_h(\cdot)$ we integrate up to $h = 1$ seconds. Reflective boundary conditions are applied after every Euler-Maruyama step. We run the Adam optimizer with standard parameters ($\beta_1 = 0.9, \beta_2=0.99$) and an initial `learning rate' of $10^{-1}$. We decrease the learning rate by a factor $10$ every $100$ epochs, for a total of $300$ epochs. The initial distribution is a truncated Gaussian with a mean $5$ and a standard deviation of $2$.

As we see in Figure~\ref{fig:w2_chemotaxis}, the Wasserstein-Adam optimizer reliably converges to a final loss of $2 \ 10^{-5}$. The loss gradient also decreased to a similar value, indicating strong convergence. The histogram of the optimized particles (right figure, orange) also matches well with the analytic invariant distribution. Finally, we see in the figure on the left that the Wasserstein-Adam optimizer reached this steady-state already after about $150$ epochs - or a total simulation time of $150$ seconds. This is because the objective function is evaluated only once per Adam epoch. Compared to computing the steady-state by long-time evolution of~\eqref{eq:chemotaxis}, which takes about $500$ seconds of in-simulation time, the Wasserstein-Adam method is much faster.

\begin{figure}[h]
    \centering
    \begin{subfigure}[b]{0.49\textwidth}
        \centering
        \includegraphics[width=\textwidth]{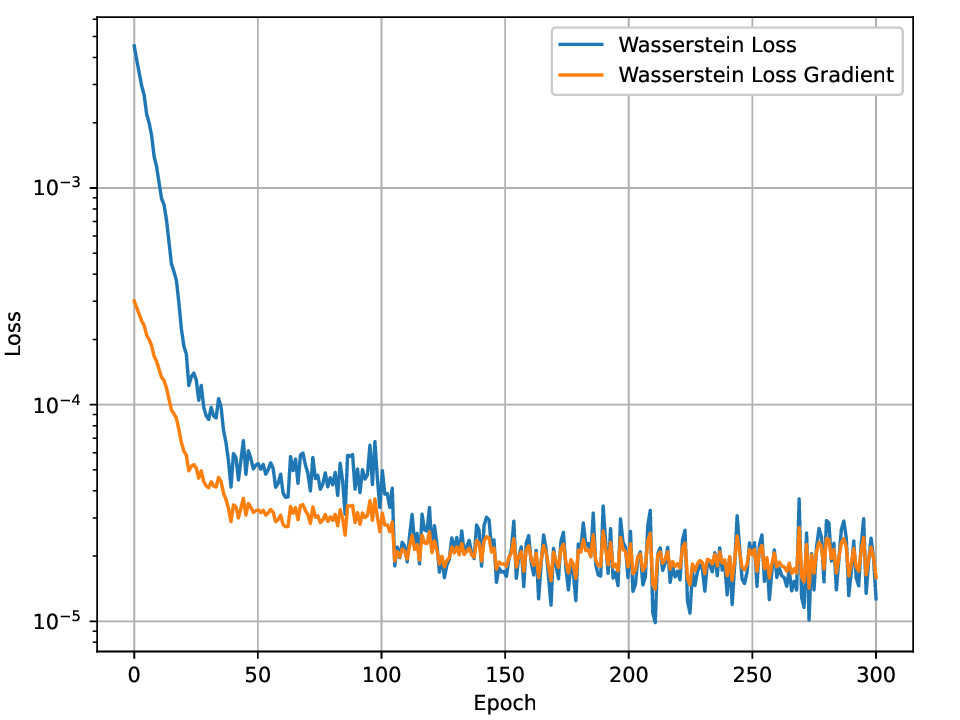}
    \end{subfigure}%
    \begin{subfigure}[b]{0.49\textwidth}
        \centering
        \includegraphics[width=\textwidth]{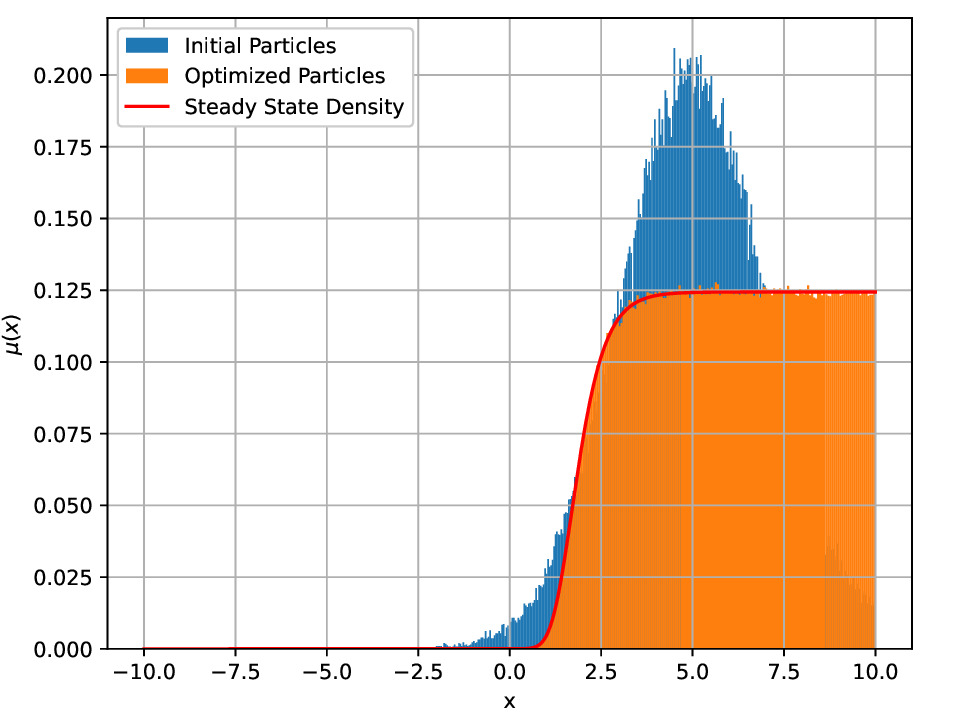}
    \end{subfigure}
    \caption{Numerical results for the Wasserstein-Adam optimizer on the chemotaxis model. Left: Wasserstein loss (blue) and gradient-norm (orange) per epoch. Right: Histogram of the initial particles in blue, the Wasserstein-Adam optimized particles in orange, and the analytic steady-state density (see equation~\eqref{eq:chemotaxis_ss}) in red.}
    \label{fig:w2_chemotaxis}
\end{figure}

\subsubsection{Economic Agents} \label{subsubsec:w2_agents}
The Wasserstein-Adam method also works well for nonlinear models. Let us consider an example of economic agents trading stocks. Each agent is represented by a value $X_i(t) \in (-1,1)$ that gives the tendency to buy ($X_i =1$) or sell ($X_i = -1$) a certain stock. More details on the model can be found in~\cite{fabiani2024task}. We use the timestepper available at~\cite{evangelou_agent_based} to integrate the agents up to $h = 1$ seconds at a time. This timestepper is unbiased for the dynamics of the agents' dynamics and uses the discretization presented in~\cite{fabiani2024task}. We again use an Adam optimizer with an initial learning rate of $10^{-1}$ and with standard parameters to find the steady-state distribution of the particles. The learning rate is decreased by a factor $10$ every $100$ epochs for a total of $300$ epochs. The initial distribution is a Gaussian with zero mean and a standard deviation of $0.1$. The steady-state distribution is wider and approximately normal.

Figure~\ref{fig:w2_agents} displays the Wasserstein loss and gradient on the left, as well as optimized agents compared to the steady-state distribution on the right. The initial distribution of agents is shown too. We observe an excellent match between the histogram of optimized agents and the invariant distribution. The loss stagnates near $2 \ 10^{-6}$, nearly four orders of magnitude lower than the initial loss, demonstrating a clear convergence.

\begin{figure}[h]
    \centering
    \begin{subfigure}[b]{0.49\textwidth}
        \centering
        \includegraphics[width=\textwidth]{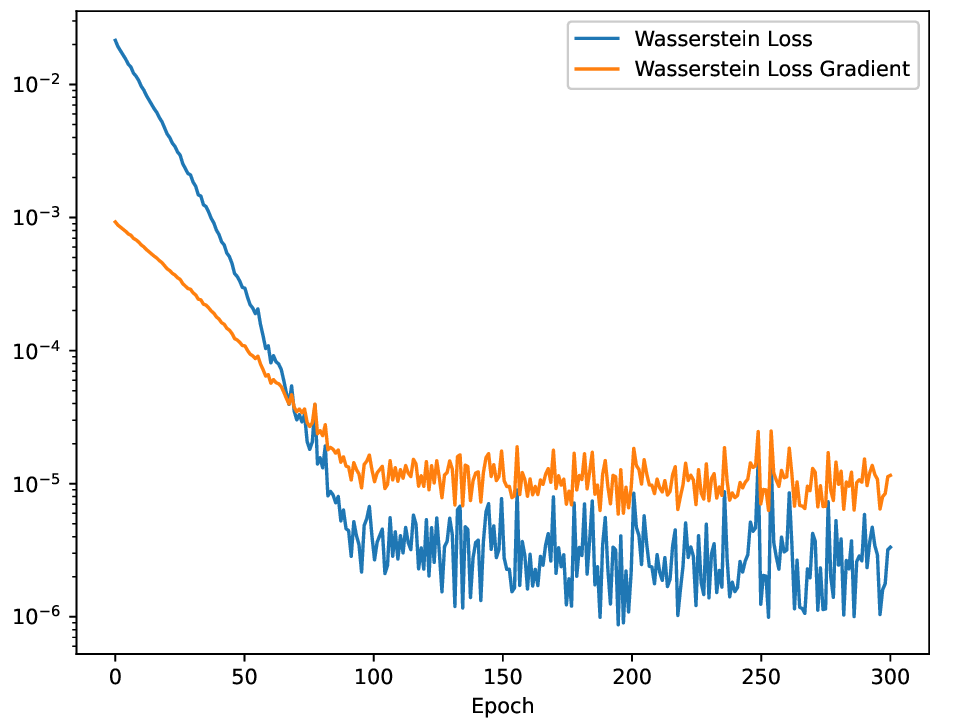}
    \end{subfigure}%
    \begin{subfigure}[b]{0.49\textwidth}
        \centering
        \includegraphics[width=\textwidth]{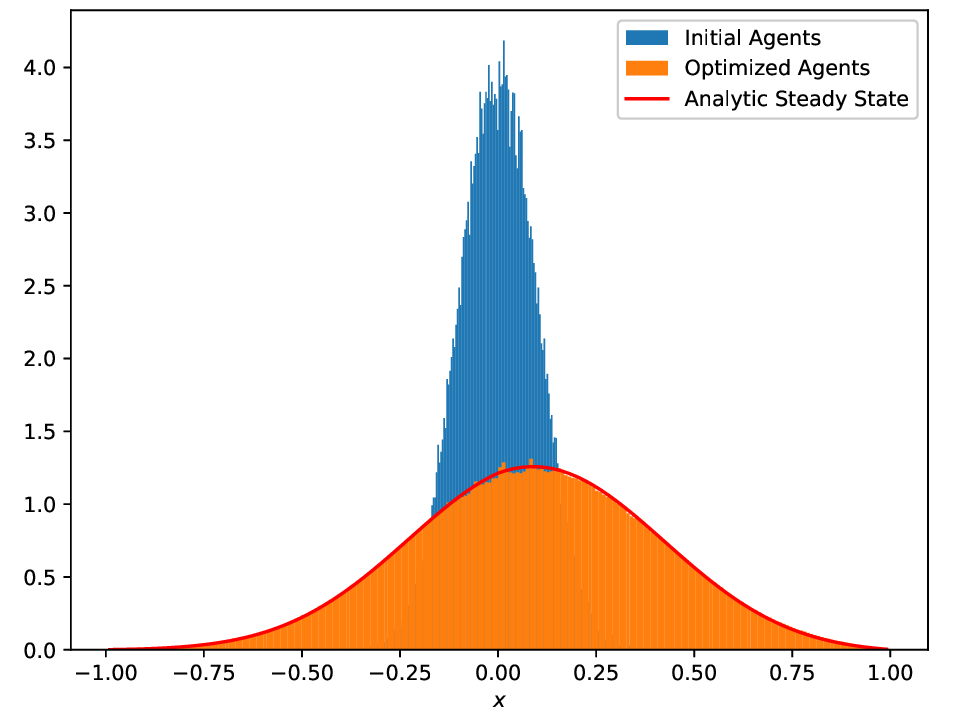}
    \end{subfigure}
    \caption{Numerical results for the Wasserstein-Adam optimizer on the economic agents model. Left: Wasserstein loss (blue) and gradient-norm (orange) per epoch. Right: Histogram of the initial particles in blue, the Wasserstein-Adam optimized particles in orange, and the analytic steady-state density (see Figure~\ref{fig:agent_pde_ss}) in red.}
    \label{fig:w2_agents}
\end{figure}

\subsubsection{The Half-Moon Potential} \label{subsubsec:w2_halfmoon}
In multiple dimensions, one cannot exploit the monotone rearrangement used in 1D (sorting), so computing $W_2$ requires solving a discrete assignment problem between the two point clouds. We use the linear assignment solver, whose worst-case complexity is $\mathcal{O}(M^3)$ for $M$ paired points; applied to the full cloud, this would be $\mathcal{O}(N^3)$. To keep costs manageable, we adopt mini-batches of size $B$: 
each step propagates $B$ particles through the timestepper and solves an $\mathcal{O}(B^3)$ assignment problem. These batches are drawn uniformly from the original $N$ particles, and 
We process $N/B$ batches per iteration, yielding a total $\mathcal{O}(NB^2)$ complexity. The smaller the batch size, the smaller the computational cost. 
The rest of the Wasserstein–Adam flow is unchanged from 1D: we form $Y_b=\varphi_h(X_b)$, compute an optimal match $\Pi_b^\star$ on each batch $b$, use the detached gradient $\partial_X \tfrac{1}{2} W_2^2(X_b, Y_b) = X_b-\Pi_b^\star Y_b$, and update the particles $X_b$ with Adam.

We note that the effective result of this is equivalent to solving a \emph{constrained} optimal assignment problem, where we require the \(N\times N\) permutation matrix \(\Pi\) to be block-diagonal (i.e., the $N/B$ diagonal $B\times B$ blocks are the permutation matrices \(\Pi^*_b\)).
This permutation matrix is in general not the exact optimal coupling for the overall \(N\times N\) assignment problem, but if each batch is chosen randomly, the batched coupling will be a close approximation of the optimal coupling --- in particular, it still gives a good descent direction for use in Adam.

For this example, we consider a two-dimensional half-moon potential. The potential energy function $U(x,y)$ is given by
\begin{equation} \label{eq:halfmoonpotential}
    U(x, y) = A \left(r(x,y) - R\right)^2 + B \exp\left( - \alpha (y - y_s)\right)
\end{equation}
with $A = 2, B= 0.5, R = 2, \alpha=1.5$, and $y_s=-0.5$. The function $r(x,y) = \sqrt{x^2+y^2}$ measures the distance from the origin. A two-dimensional color plot of the corresponding steady-state distribution
\begin{equation} \label{eq:halfmoon_dist}
    \mu(x,y) = Z^{-1} \exp\left(-U(x,y)\right)
\end{equation}
is shown in Figure~\ref{fig:halfmoonpotential}.

\begin{figure}
    \centering
    \includegraphics[width=0.8\linewidth]{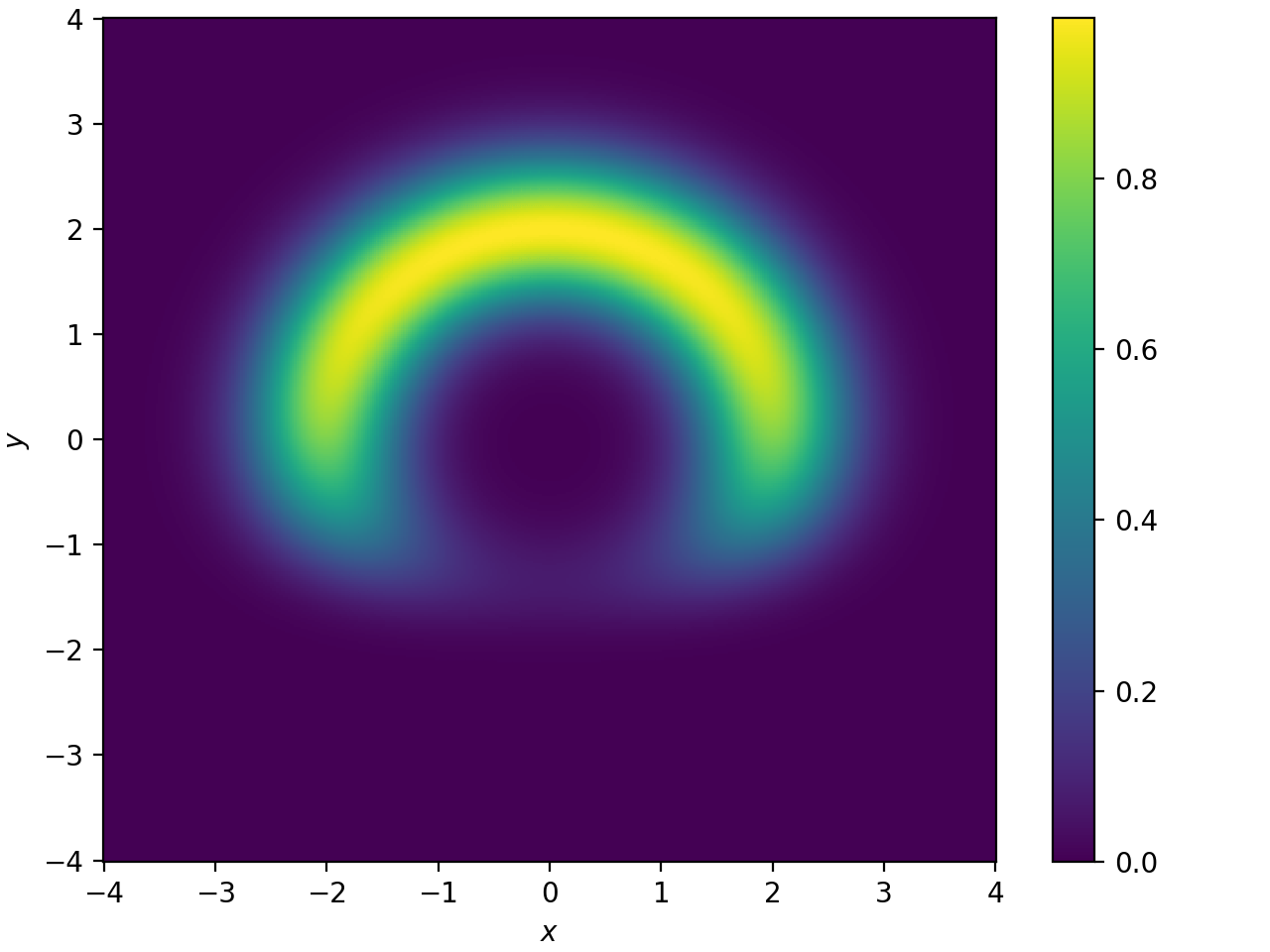}
    \caption{A color plot of the half-moon potential, see equation~\eqref{eq:halfmoonpotential} for more details.}
    \label{fig:halfmoonpotential}
\end{figure}

To compute the steady-state distribution using the Wasserstein-Adam method, we construct an overdamped Langevin dynamics of the form
\begin{equation} \label{eq:langevin_halfmoon}
    d(X,Y) = -\nabla U(X,Y) dt + \sqrt{2} dW_t
\end{equation}
where $W_t$ is a two-dimensional Brownian motion with independent components. The steady-state distribution of this dynamics is precisely~\eqref{eq:halfmoon_dist}. We also implement reflective boundary conditions (e.g., Neumann) to keep all particles within the domain $[-4,4] \times [-4,4]$. We discretize this overdamped Langevin dynamics using the Euler-Maruyama timestepper with step size $10^{-3}$, and integrate up to time $h = 0.1$ seconds. The initial distribution is a standard bivariate Gaussian, which has minimal overlap with~\eqref{eq:halfmoon_dist}.

To compute the Wasserstein loss, we divide the $N=10^5$ particles into batches of size $1000$. For each batch, we propagate the particles through the timestepper and evaluate the loss using the OT coupling algorithm. We feed these batches to the Adam optimizer, which uses an initial learning rate of $10^{-2}$, decreasing by a factor of $10$ every $100$ epochs, up to $300$ epochs. 

The optimization results are shown in Figure~\ref{fig:w2_halfmoon}. The left panel reports the objective and gradient-norm trajectories: the initial $\tfrac12 W_2^2$ of $0.51233$ decreases monotonically to $1.54\times10^{-2}$, effectively reaching the stochastic noise floor, and the loss gradient decays in tandem, indicating stable convergence. The right panel displays a 2D histogram (color map) of the optimized particles, which shows an excellent match to the analytic steady-state distribution from figure~\ref{fig:halfmoonpotential} with no visible systematic bias.

\begin{figure}[h]
    \centering
    \begin{subfigure}[b]{0.49\textwidth}
        \centering
        \includegraphics[width=\textwidth]{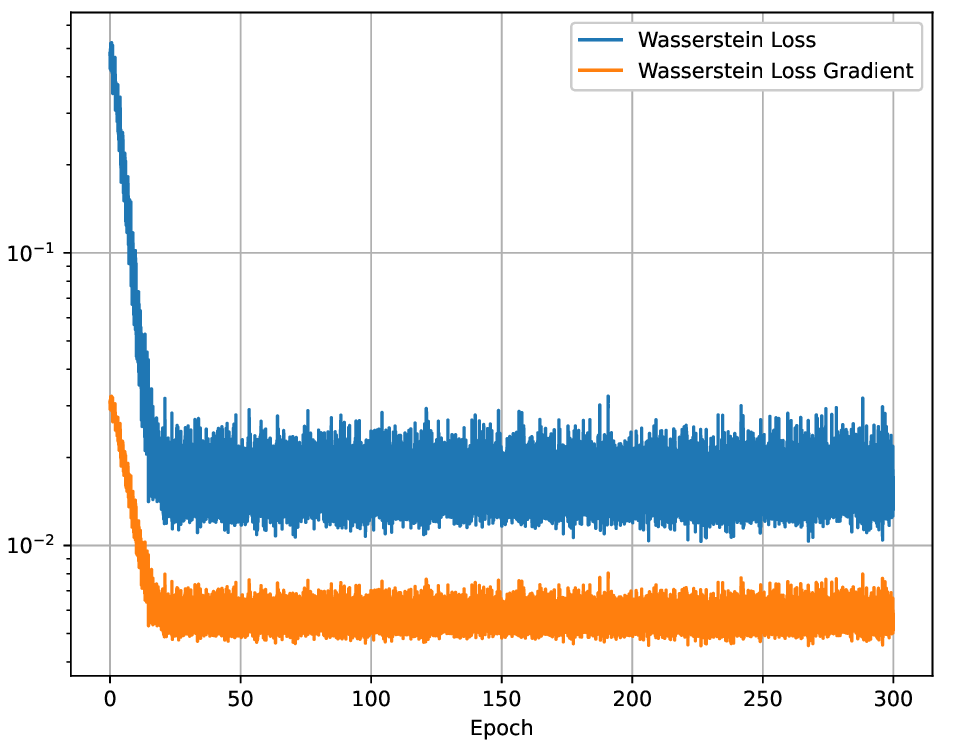}
    \end{subfigure}%
    \begin{subfigure}[b]{0.49\textwidth}
        \centering
        \includegraphics[width=\textwidth]{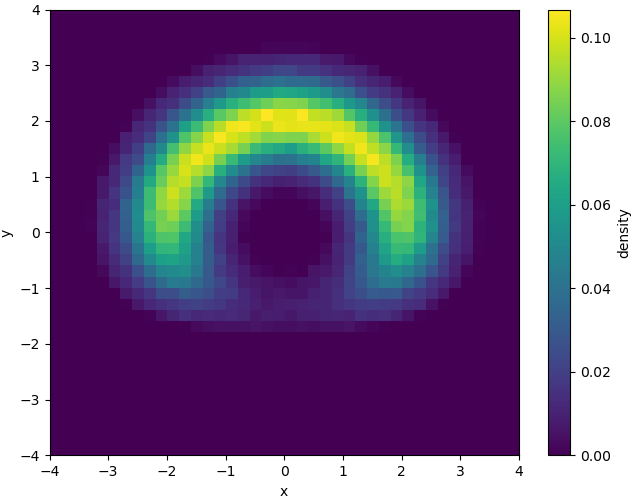}
    \end{subfigure}
    \caption{Numerical results for the Wasserstein-Adam optimizer on the half-moon potential. Left: Wasserstein loss (blue) and gradient-norm (orange) per epoch. Right: Colormap of the 2D histogram of the optimized particles in orange.}
    \label{fig:w2_halfmoon}
\end{figure}

\begin{figure}[h]
    \centering
    \begin{subfigure}[b]{0.49\textwidth}
        \centering
        \includegraphics[width=\textwidth]{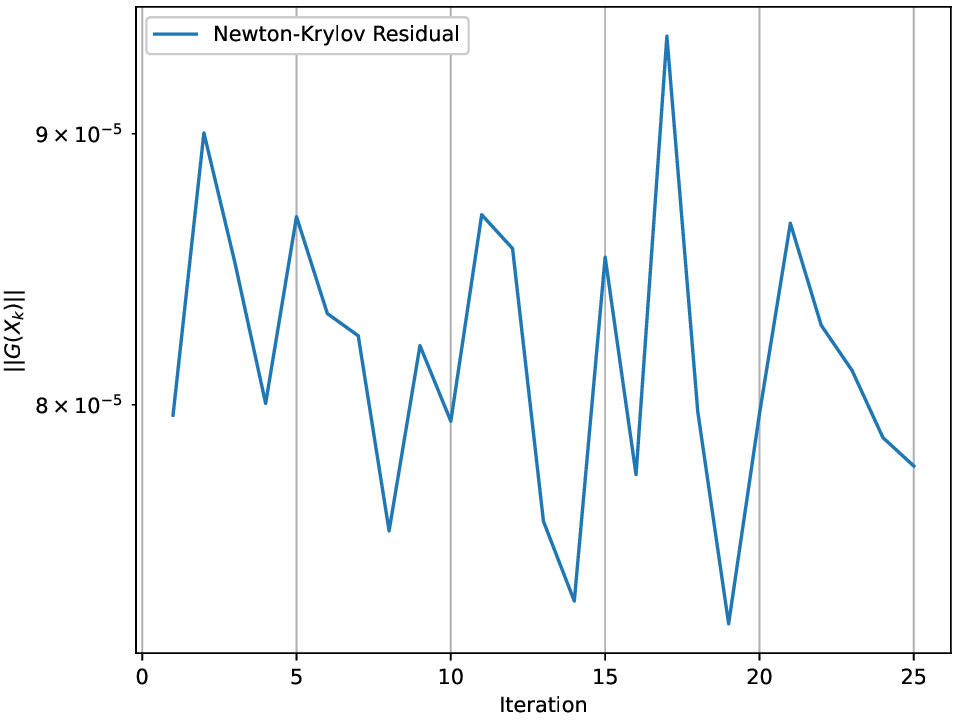}
    \end{subfigure}%
    \begin{subfigure}[b]{0.49\textwidth}
        \centering
        \includegraphics[width=\textwidth]{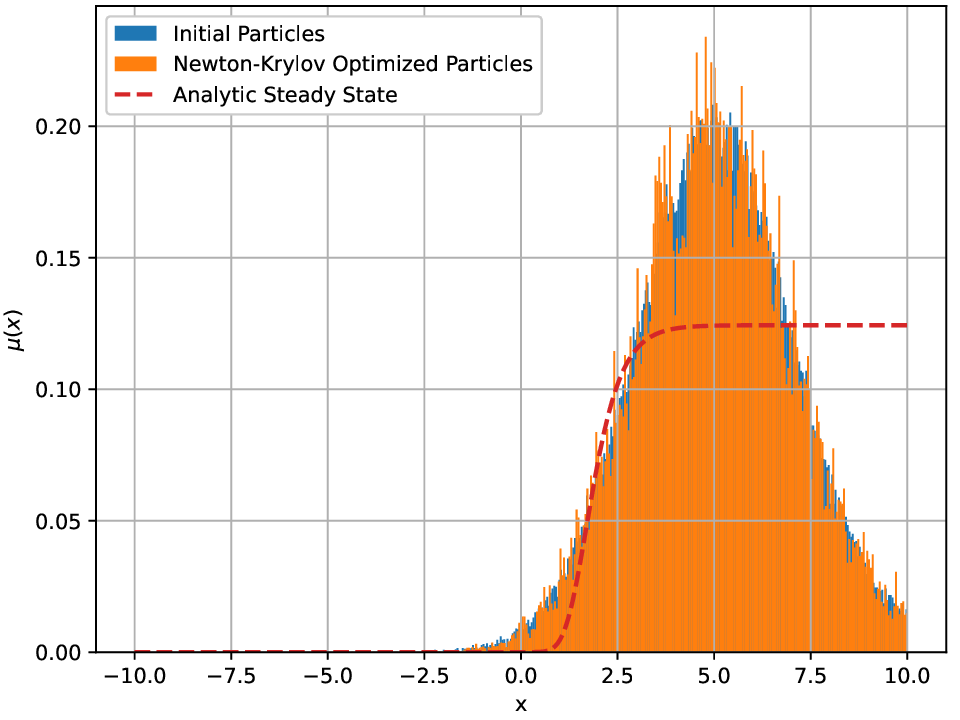}
    \end{subfigure}
    \caption{Numerical results for the Newton--Krylov optimizer for bacterial chemotaxis. Left: Norm of the objective function, $\norm{\tilde{F}(X_k)}$ per iteration. Right: Initial particles (blue), optimized particles (orange) and analytic steady-state distribution (dashed red).}
    \label{fig:nk_particles_chemotaxis}
\end{figure}

\section{Newton--Krylov Framework} \label{sec:nk}
First-order optimizers typically converge slowly to the steady state, as they take only small steps in the direction of the local gradient. As a result, they require (relatively) many iterations and repeated evaluations of the Wasserstein objective and gradient, each of which involves the costly computation of an optimal transport map. This makes first-order methods particularly expensive in Wasserstein-based formulations. We want to recover the second-order performance of Newton–Krylov by solving
\begin{equation}
F(X)=\partial_X W_2^2\!\big(X,\ \phi_h(X)\big)=0
\end{equation}
and approximating Jacobian–vector products as
\begin{equation}
DF(X) \cdot  v \approx  \frac{F(X+\varepsilon v)-F(X)}{\varepsilon}.
\end{equation}
In practice this is ill-posed and noisy for particle OT. The mapping $X\mapsto \partial_X W_2^2(X,\phi_h(X,\xi)$ is only piecewise smooth: as particles move, the optimal transport coupling changes combinatorially, creating kinks where $DF(X)$ does not exist. Difference quotients then mix values across distinct couplings, producing large, non-vanishing errors that blow up as $\varepsilon \to 0$. Stochasticity compounds the problem: tiny perturbations $X \mapsto X+\varepsilon v$ can reshuffle assignments in the empirical OT, injecting high-variance jumps into $F(X+\varepsilon v)-F(X)$. Krylov iterations driven by such JVPs become unstable, and line searches lose reliability. We perform a more detailed error analysis of the Newton--Krylov method for noisy objective functions in section~\ref{subsec:nk_noisy}, before moving on to using Newton--Krylov on smooth representations of particle codes in section~\ref{sec:particles_to_smooth}.

The idea of Newton's method is to iteratively update \(u_{k+1} = u_k + s_k\) by choosing \(s_k\) such that the first order approximation of \(\psi(u_{k+1}) \approx 0\). Explicitly, we want \(s_k\) such that \(\psi(u_k) + D\psi(u_k)s_k = 0\).
The main idea behind the Newton--Krylov method is to approximate the action of the Jacobian of the objective function $\nabla \psi(u_k)$, using finite differences with step size $\varepsilon$
\begin{equation} \label{eq:fd_2}
    D \psi(u_k) \cdot  v \approx \frac{\psi(u_k + \varepsilon v) - \psi(u_k)}{\varepsilon}.
\end{equation}
Since this matrix-free formulation can only approximate matrix-vector products, and not the full Jacobian directly, we must use iterative Krylov methods to solve the linear system
\begin{equation} \label{eq:linear_newton}
    D\psi(u_k) s_k = -\psi(u_k)
\end{equation}
to obtain the next Newton guess $u_{k+1} = u_k + s_k$. Because the Jacobian is not generally symmetric, we will use the GMRES method for solving~\eqref{eq:linear_newton}.

The authors of~\cite{brown2008using} performed an extensive error analysis of the Newton--Krylov method; we will repeat their main findings here. There are three main sources of error in the Newton--Krylov method: (1) approximation and rounding errors in the approximate Jacobian~\eqref{eq:fd_2}, (2) tolerance $\eta_k$ used to solve~\eqref{eq:linear_newton}, and (3) rounding errors in other calculations in addition to these. Each of these errors can be controlled. Indeed, one typically uses $\varepsilon \sim \mathcal{O}\left(\sqrt{\varepsilon_{\text{mach}}}\right)$ where $\varepsilon_{\text{mach}} = 2.2 \ 10^{-16}$ is the machine precision in double precision floating point arithmetic. This limits the first error contribution to the same bound of
$\mathcal{O}\left(\sqrt{\varepsilon_{\text{mach}}}\right)$. For the second error contribution, the authors of~\cite[Theorem 2.12]{brown2008using} proved that if the sequence of GMRES tolerances $\eta_k$, defined such that
\begin{equation}
    \norm{D\psi(u_k) s_k + \psi(u_k)} \leq \eta_k \norm{\psi(u_k)}
\end{equation}
decreases to $0$, then the Newton--Krylov method converges linearly to the exact root of $\psi$, i.e., $u_k \to u_{\infty} = u^*$ where $u^*$ is the exact solution. Their analysis remains true whenever the approximation and rounding errors in~\eqref{eq:fd_2} remain minimal; that is, they are of order $\mathcal{O}\left(\sqrt{\varepsilon_{\text{mach}}}\right)$.

Of course, in practice, the tolerances $\eta_k$ do not decrease in subsequent nonlinear Newton--Krylov iterations; rather, they remain constant at $\eta_k = \eta$ for all $k \geq 1$. In this case, the Newton--Krylov method will always make a deterministic error of size
\begin{equation} \label{eq:backward_nk_error}
    \norm{\psi(u_\infty)} \leq C\left( \sqrt{\varepsilon_{\text{mach}}} + \limsup_{k} |\eta_k| + \varepsilon_{\text{mach}}\right).
\end{equation}
This result can be proven as a consequence of Theorem 2.12 in~\cite{brown2008using}. Here, $C$ is a constant independent of $\varepsilon_{\text{mach}}$ and $\eta_k$ but not of $\psi$ and $h$. The final term in equation~\eqref{eq:backward_nk_error} of $\varepsilon_{\text{mach}}$ is due to (3): general rounding errors in the remaining calculations. Finally, since the condition number of a simple root is $\norm{D\psi(u^*)^{-1}}$, we also obtain a bound for the forward error
\begin{equation}
 \norm{u_\infty - u^*} \leq C \norm{D\psi(u^*)^{-1} }\left( \sqrt{\varepsilon_{\text{mach}}} + \limsup_{k} |\eta_k| + \varepsilon_{\text{mach}} \right)
\end{equation}
in floating point arithmetic.

\subsection{Extending the Error Analysis to Stochastic Timesteppers} \label{subsec:nk_noisy}
We aim for faster convergence to steady state than is attainable with first-order descent methods through gradient-based optimizers. A natural idea is to invoke second-order optimization, yet in the particle setting the underlying derivatives are not fully well defined. The basic idea would be to start from the Wasserstein objective function~\eqref{eq:discrete_W2} and solve for a zero (truncated) gradient 
\begin{equation} \label{eq:nk_particle_objective}
    F(X) = \partial_X \tfrac{1}{2} W_2^2\left(X, \phi_h(X)\right) = X - \Pi^\star \phi_h(X),
\end{equation}
where $\Pi^\star$ is the optimal transport map from $X$ to $Y = \phi_h(X)$. For any finite number of particles $N$ there will be a noise term for each evaluation of the Wasserstein objective function. Indeed,
\begin{equation}
    \tfrac{1}{2} W_2^2(X, \phi_h(X)) = \tfrac{1}{2}W_2^2(\mu, \text{law}(Y)) + \frac{\xi}{\sqrt{N}},
\end{equation}
where $\mu$ is the distribution of samples $X$. 
For illustrative purposes, let us suppose that the noise term $\xi$ is independent of $X$.
The gradient objective function~\eqref{eq:nk_particle_objective} will then also include a noise term
\begin{equation} \label{eq:nk_particle_objective_noise}
    \tilde{F}(X) = F(X) + \frac{\zeta}{\sqrt{N}}
\end{equation}
where the contributions per-component of $\zeta \in \mathbb{R}^N$ add up to $\xi$.

Any second-order optimization method will require Jacobians of the objective function in equation~\eqref{eq:nk_particle_objective}. In the Newton--Krylov method, the action of the Jacobian of $F$ in the direction of $v$ is approximated by directional finite differences. Under the assumption that subsequent evaluations of $F$ introduce independent noise terms $\zeta_1$ and $\zeta_2$, approximating the action of the Jacobian $D F$ on a vector $v$ introduces an additional variance term that scales unfavorably as the step size $\varepsilon \to 0$
\begin{equation} \label{eq:nk_particles_Gv}
D \tilde{F}(X)v \approx \frac{\tilde{F}(X+\varepsilon v) - \tilde{F}(X)}{\varepsilon} = \frac{F(X+\varepsilon v) - F(X)}{\varepsilon} + \frac{\zeta_1 - \zeta_2}{\varepsilon\sqrt{N}}.
\end{equation}
Since $\zeta_1$ and $\zeta_2$ are independent, for example, through independent random numbers in the timestepper, and $\varepsilon$ is typically small (say $\sqrt{\varepsilon_{\text{mach}}} \approx 10^{-8}$), the resulting noise term $(\zeta_1 - \zeta_2) / \varepsilon\sqrt{N}$ will dominate the Jacobian-vector product. The estimated Jacobian therefore carries essentially no usable descent information, and the Newton–Krylov method will degenerate into stochastic oscillations rather than a systematic search direction, failing to converge toward the true steady state. Continuing the error analysis from Section~\ref{sec:nk}, the asymptotic error of the Newton--Krylov method is governed by a fundamental bias–variance trade-off,
\begin{equation} \label{eq:nk_particle_error}
\norm{X_\infty - X^\star} \leq C\norm{F(X^\star)^{-1}}\left( \varepsilon + \limsup_{k} |\eta_k| + \varepsilon_{\text{mach}} + \frac{2\sigma}{\varepsilon \sqrt{N}} \right)
\end{equation}
where $\sigma$ is the standard deviation of $\xi$. The above equation highlights the challenge of second-order methods on the level of particles: although they carry the potential for rapid convergence, the noise inherent in particle-based Jacobians fundamentally limits their accuracy. However, we can use equation~\eqref{eq:nk_particle_error} to determine $\varepsilon$ so that the approximation and noise terms are balanced. It can be seen that the minimal total error is achieved when $\varepsilon = \mathcal{O}\left(N^{-1/4}\right)$. As an example, for $N=10^4$, we would need to use $\varepsilon = 0.1$ -- which is much larger than $\sqrt{\varepsilon_{\text{mach}}}$ in the absence of noise.

However, even with this optimal finite-difference step size, Newton--Krylov at the particle level cannot discern any gradient information from~\eqref{eq:nk_particles_Gv}. Let us consider the bacterial chemotaxis example again from section~\ref{subsubsec:w2_chemotaxis}. In Figure~\ref{fig:nk_particles_chemotaxis} we show the Newton--Krylov optimization result after $100$ steps with $N=10^5$ particles and $\varepsilon = 0.1$. The optimized particles barely moved from their initial distribution. Figure~\ref{fig:nk_particles_chemotaxis} shows the norm of the objective function~\eqref{eq:nk_particle_objective_noise} in all iterations, revealing that the error remains essentially unchanged in expectation. The reason for this behavior is that although each application of the timestepper moves the particles closer to the steady state, the subsequent Newton–Krylov update perturbs them in essentially random directions. As a result, the timestepper is forced to repeatedly recover the same progress, preventing any net convergence.

\subsection{Determining $h$ from the Spectral Gap}
Beyond our Newton–Krylov error analysis, it is instructive to examine how information about the steady state $u^*$ is encoded in timestepper $\phi_h$ and residual $\psi(u) = u-\phi_h(u)$. Linearization of the PDE
\begin{equation}
    \partial_t u_t = f(u_t)
\end{equation}
around $u^*$ gives $J = D f(u^*)$ with eigenvalues $\lambda_i$. This linearization of the flow map $\phi_h$ is then $D \phi_h(u^*) = \exp\left(h J\right)$ so that
\begin{equation}
    D \psi(u^*) = I - \exp\left(h J\right).
\end{equation}
Let $\mu_i$ be the eigenvalues of $D \psi(u^*)$. This mapping transfers the continuous-time stability spectrum of $f$ into a discrete-time residual spectrum: eigenvalues with $\operatorname{Re}\lambda_i \approx 0$ (slow or neutral modes) yield $|\mu_i|\approx 0$, while strongly damped modes $(\operatorname{Re}\lambda_i \ll 0)$ are pushed toward $|\mu_i|\approx 1$.  As $h$ increases, a clear spectral gap typically opens between these two clusters.  Choosing $h$ too small leaves the modes entangled; choosing it too large pushes all $\mu_i$ towards $1$ and degrades conditioning. The optimal integration window therefore balances the resolution of slow modes with numerical stability and cost.

Let $u$ be the current Newton--Krylov guess, a point assumed close to the steady-state solution $u^*$. It is safe to say that the eigenvalues and eigenvectors of $D\phi_h(u)$ will not deviate from $\lambda_i$ and $v_i$ respectively too much. Also let $\lambda_1$ be the eigenmode corresponding to the dominant, near steady-state dynamics. For a stable PDE solution, this means $0 > \text{Re}(\lambda_i) > \text{Re}(\lambda_i)$ for all $i > 1$. By perturbing the current solution $u$ in the direction of $v = \sum_{i=1}^n a_i v_i$, the flow map $\phi_h(u)$ will change to first order by
\begin{equation} \label{eq:Jac_expansion}
    D \phi_h(u) v = D \phi_h(u) \left(\sum_{i=1}^n  a_i \phi_h(u) v_i \right) = \sum_{i=1}^n a_i \exp(h \lambda_i) v_i,
\end{equation}
where we used the fact that $v_i$ are (approximate) eigenvectors of $D\phi_h(u)$. In this expansion, the factor $\exp(h \lambda_1)$ will be larger than all other $\exp(h \lambda_i)$ in real part. Similar to the analysis of the power iteration, the above expansion can be rewritten as the dominant term plus high-frequency corrections
\begin{equation}
     D \phi_h(u) v =  \exp(h \lambda_1)\left(a_1 v_1 + \sum_{i=2}^n a_i \frac{\exp(h \lambda_i)}{ \exp(h \lambda_1)} v_i \right).
\end{equation}
When the integration time $h$ is large enough, the other terms in this expansion will be small compared to the first term $\exp(T \lambda_1) a_1 v_1$, which represents the steady state solution. 

To ensure that the Newton–Krylov method can discern meaningful descent information from the Jacobian action  $D\phi_h(u) \cdot v$, rather than being dominated by high-frequency `noise' from the remaining terms in~\eqref{eq:Jac_expansion}, a high signal-to-noise ratio is required. Mathematically, we require
\begin{equation}
    \left|\frac{\exp(h \lambda_i)}{\exp(h \lambda_1)} \right|\ll 1
\end{equation}
and this provides a criterion for choosing the minimal integration time $h$. Defining the signal-to-noise ratio as $1/\eta$ with $\eta \ll 1$, we desire
\begin{equation*}
    \exp\left(h \left(\text{Re}(\lambda_i) - \text{Re}(\lambda_1)\right)\right) \leq \eta
\end{equation*}
 for all $i > 1$, and therefore
\begin{equation}
    h \geq \frac{\ln \eta}{\text{Re}(\lambda_2) - \text{Re}(\lambda_1)}
\end{equation}
because $\text{Re}(\lambda_i) \leq \text{Re}(\lambda_2) < \text{Re}(\lambda_1) < 0$. In practice, we typically want a signal-to-noise ratio of $1/\eta \geq 10$, and the above equation gives the minimal value for the total integration time $h$. Increasing $h$ beyond this bound will result in more expensive computations with little benefit of increased speed of convergence to the steady-state solution.

\subsection{Newton--Krylov for Mean-Field Equations}
As a first step towards using the Newton--Krylov method for calculating steady-state solutions, let us look at the case when the underlying model for the particle timestepper is a partial differential equation, specifically a mean-field or Fokker-Planck equation.
We prefer the terminology of the mean-field equation since the Fokker-Planck equation is only valid for linear models; however, we also consider nonlinear stochastic models here.

We should note that our ultimate goal is to compute steady-state distributions of particle timesteppers. The mean-field PDE, although a useful mathematical concept, is not immediately known or available for most systems and particle codes. In the few cases where it is known, the mean-field model is usually an approximation through some kind of closure. However, in the settings where the mean-field model is known and exact - like the linear Fokker-Planck equation - steady-state distributions are directly encoded, and we can compute it using a Newton--Krylov method. We investigate this idea in this section.

Suppose we have a collection of particles or agents $\{X_i\}_{i=1}^n$ that follow a stochastic process. We will discuss two examples below. In the limit of $n \to \infty$, the particles or agents are distributed according to a time-dependent probability distribution $\mu_t(x)$, where $x$ represents the spatial locations. This probability distribution, in turn, follows a deterministic and coarse mean-field equation of the form
\begin{equation} \label{eq:mean-field}
    \partial_t \mu_t(x) = F(\mu_t(x))
\end{equation}
with the appropriate boundary conditions. On this level, all stochasticity has been removed, and we can simply use the regular Newton--Krylov method to compute the steady-state distributions. Most mean-field models never have their right-hand sides coded, so we will rely on a coarse timestepper $\Phi_h(\mu_t)$ that progresses the current distribution $\mu_t$ over a time window of size $h$ to $\mu_{t+h}$. Analogous to equation~\eqref{eq:psi}, we will solve
\begin{equation}
    \Psi(\mu) = \mu - \Phi_h(\mu) = 0
\end{equation}
for $\mu$ --- the steady-state distribution. We will demonstrate this approach on two examples from last section: the linear chemotaxis model (section~\ref{subsubsec:chemotaxis}) and the nonlinear economic agents model (section~\ref{subsusbec:agentmodel}).

\subsubsection{Bacterial chemotaxis model} \label{subsubsec:chemotaxis}
The coarse Fokker-Planck equation for the spatial distribution $\mu(x,t)$ of the bacteria from section~\ref{subsubsec:w2_chemotaxis} reads 
\begin{equation} \label{eq:keller-segel}
 \partial_t \mu = -\partial_x(\chi(S(x))S_x(x)\mu) + D\partial_{xx}\mu,
\end{equation}
and describes how the concentration of bacteria changes over time. This equation is sometimes known as the Keller-Segel model for chemotaxis~\cite{stevens2000derivation}. The correct no-flux boundary conditions are $J(\pm L) = 0$ with flux
\[
J(x) = \chi(S(x))S_x(x)\mu(x,t) - D \mu_x(x,t).
\]
It admits a unique steady-state distribution of the form
\begin{equation} \label{eq:chemotaxis_ss}
 \mu(x) = Z^{-1}\exp\left(\frac{1}{D}\int^{S(x)}_{-1} \chi(S) dS\right),
\end{equation}
where $Z$ is the normalization constant. See Appendix~\ref{app:chemotaxis} for a short derivation of~\eqref{eq:chemotaxis_ss}.

In our example, $S(x) = \tanh(x)$ and $\chi(S) = 1 + \frac{1}{2}S^2$. We compute the steady-state distribution using the Newton--Krylov method based on a finite volumes discretization of~\eqref{eq:keller-segel} with $N=1000$ equidistant grid points representing the volume centers. That is, the solution will be a vector $u \in \mathbb{R}^N$ representing the values of the steady-state distribution at the fixed volume centers. For timestepping, we use an explicit Euler time discretization over a time window of size $h = 1$ second. 

Figure~\ref{fig:chemotaxis_fp}(a) shows the steady-state distribution obtained by the Newton--Krylov method alongside the analytic formula~\eqref{eq:chemotaxis_ss}. We observe an excellent correspondence between these two distributions. Additionally, in Figure (b), we observe that the Newton--Krylov method reaches a residual of $10^{-8}$, the minimal obtainable residual according to \eqref{eq:backward_nk_error}, after only $9$ steps.
\begin{figure}[h]
    \centering
    \begin{subfigure}[b]{0.5\textwidth}
        \includegraphics[width=\textwidth]{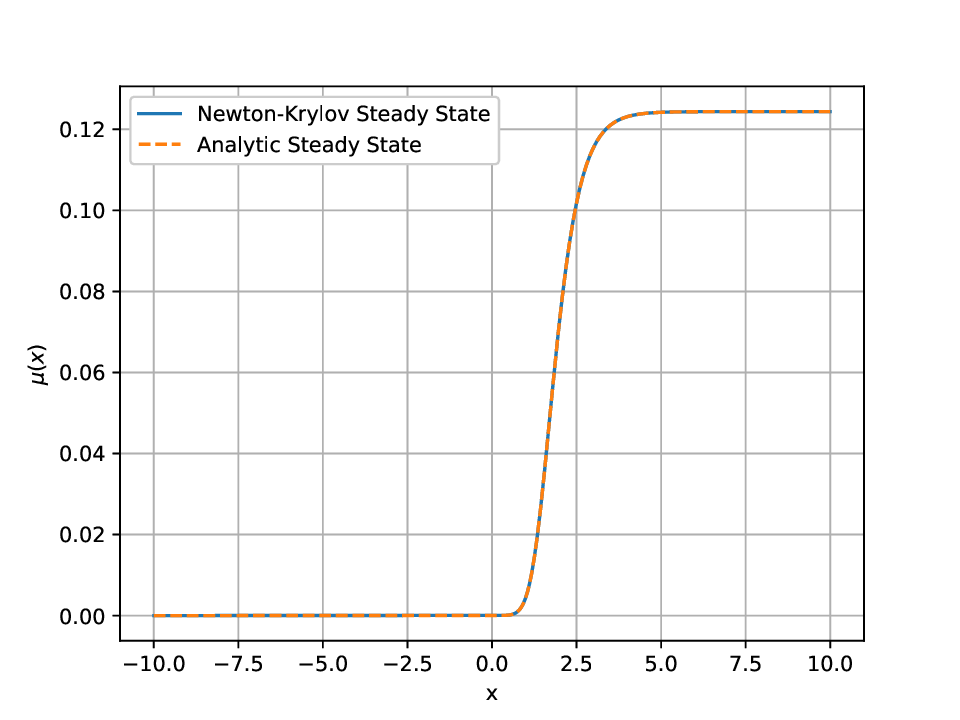}
    \end{subfigure}%
    \begin{subfigure}[b]{0.5\textwidth}
        \includegraphics[width=\textwidth]{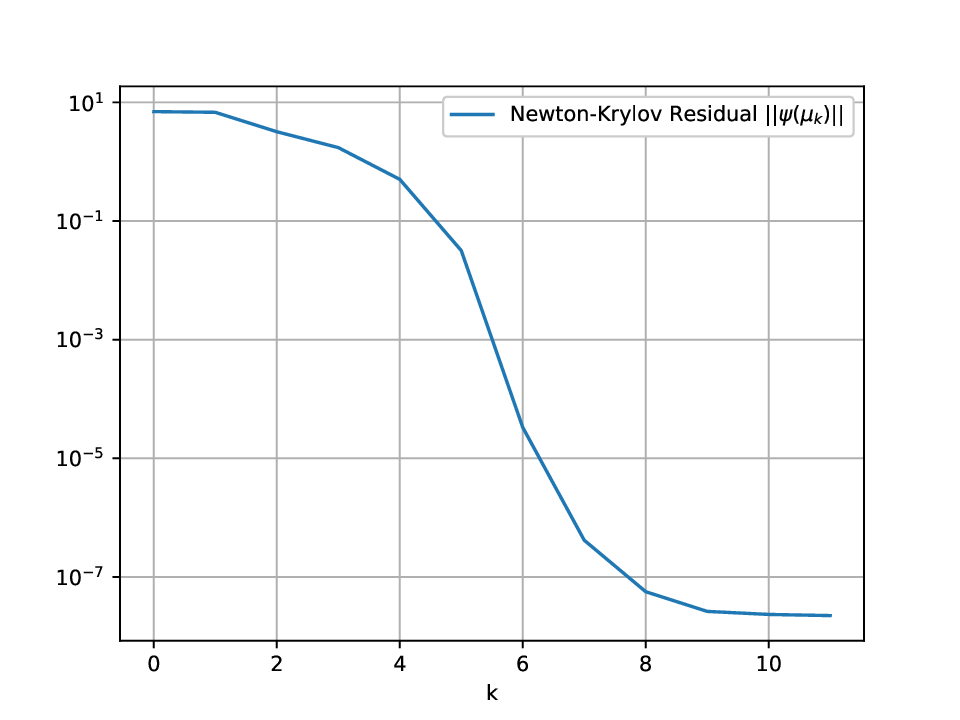}
    \end{subfigure}
    \caption{(Left) Steady-state distribution of the Keller-Segel model: Newton--Krylov (blue) vs. analytic solution (dashed orange). (Right) Newton--Krylov residual per iteration.}
    \label{fig:chemotaxis_fp}
\end{figure}

\begin{remark}
 Although the Fokker-Planck equation~\eqref{eq:keller-segel} and the $\Psi$-function are linear in $\mu_t$, the Newton--Krylov method is not guaranteed to converge to the steady-state distribution in one step, unlike the regular second-order Newton method. The reasons are twofold. First, the Newton--Krylov method does not solve the exact Jacobian system $D \psi(x_k) s_k + \psi(x_k) = 0$, rather an approximation $D \tilde{\psi}(x_k) s_k  + \psi(x_k) = 0$. Secondly, we typically use a fixed tolerance $\eta_k = \eta > 0$ for all linear system solves. So instead of converging to the real steady state distribution, the Newton--Krylov method will stay within a tolerance $\eta$ of the exact steady state; see also equation~\eqref{eq:backward_nk_error}.
\end{remark}

\subsubsection{Economic Agents Model} \label{subsusbec:agentmodel}
As a second demonstration of using Newton--Krylov to compute steady-state distributions of mean-field equations, let us reconsider the example of economic agents trading stocks. For this example, we look at the deterministic density $\mu(x,t)$ of the economic agents at location $x$ and time $t$. It can be shown that an approximate mean-field equation exists, of the form
\begin{equation} \label{eq:meanfield_agents}
    \partial_t \mu = \frac{1}{2} \sigma^2(t) \partial_{xx} \mu + \partial_x\left(b(x,t) \mu\right) + \left(J^+ + J^-\right) \delta(x).
\end{equation}
In this model, $b(x,t)$ and $\sigma^2(t)$ are the time-dependent drift and diffusivity of the agents, and $J^+$ and $J^-$ are integral operators with $\delta$ the Dirac-delta distribution. Further details on this model can be found in Appendix A of~\cite{fabiani2024task}. Importantly, this mean-field PDE is nonlinear, non-local and admits multiple steady-state distributions. 

\begin{figure}[!ht]
    \centering
    \includegraphics[width=0.7\textwidth]{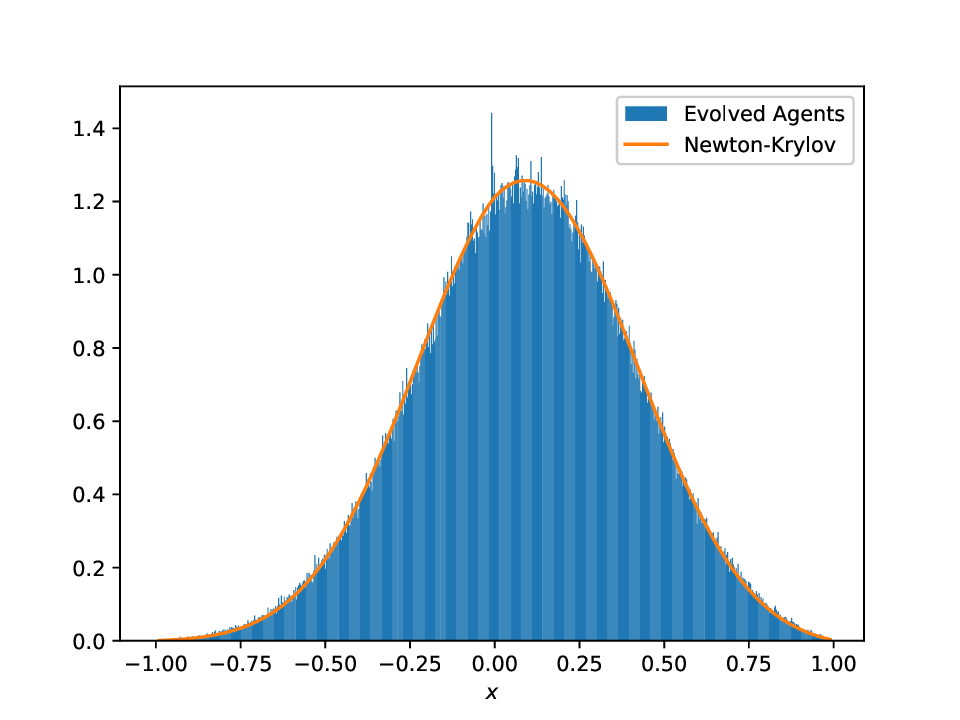}
    \caption{Steady-state distribution from Newton--Krylov versus time evolution of stochastic agents.}
    \label{fig:agent_pde_ss}
\end{figure}

For the following numerical experiment, we discretized~\eqref{eq:meanfield_agents} using a finite-volume method with $N=100$ equidistant volume centers. Our PDE timestepper uses an explicit Euler method with step size $10^{-4}$, and we integrate over a time window of size $h=11$ second. Figure~\ref{fig:agent_pde_ss} shows the steady-state distribution calculated by Newton--Krylov method in orange. It also shows the histogram distributions of the random agents obtained by timestepping the stochastic model to steady state. There is an excellent visual agreement between the two, as well as in the average agent position $\langle X_i \rangle$ of $0.0892$.

\section{From Particles to Smooth Representations} \label{sec:particles_to_smooth}
Newton–Krylov may not work on the level of noisy particles, but it can work if we add smoothness to the optimization criterion; enter the Inverse Cumulative Density Function. In one dimension, optimal transport with respect to the Wasserstein-2 distance admits an especially convenient formulation. Composing the ICDF with the CDF of the particles is exactly equivalent to computing the optimal transport map, ensuring that the geometry of the problem is preserved. Furthermore, if the ICDF is known in fixed (grid-) percentiles, we can interpolate it using splines, making the (interpolated) ICDF a smooth and monotonic representation of the empirical measure. The ICDF is continuously differentiable, so we can employ Newton-type methods to calculate steady states. Finally, evaluating the ICDF at prescribed percentiles corresponds to drawing samples from the distribution.

These combined elements make it possible to define a timestepper directly at the level of the ICDF. Starting from the current ICDF, we first sample $N$ particles by evaluating the ICDF in the (fixed) $k/N$ percentiles. Next, we propagate these samples through the stochastic particle timestepper, after which we construct the new ICDF by retaining some of the (sorted) particles. In this representation, derivatives are well defined, and Newton--Krylov iterations can be applied meaningfully, recovering the fast convergence that is lost in the purely particle-based formulation.

\subsection{The ICDF-to-ICDF Timestepper}
Given a one-dimensional probability density $\mu(x)$ on a fixed interval $[a,b]$, the cumulative density function is defined as
\begin{equation} \label{eq:1d_cdf}
 F(x) = \int_{a}^x \mu(y)dy, \ \ x \in [a,b]
\end{equation}
Because $F$ is monotonic, it is injective and inverse cumulative density $F^{-1}(p)$ exists. This ICDF is defined for $p \in [0,1]$ and $x = F^{-1}(p)$ corresponds to the $p$-th percentile of $\mu$.

For a discrete set of (sorted) particles $\{X_n\}_{n=1}^N$, the cumulative density function is piecewise continuous with jumps at the locations
\begin{equation}
 F(X_n) = \frac{n}{N}.
\end{equation}
The inverse cumulative density function is also piecewise continuous defined in the discrete percentiles $F^{-1}\left(n/N\right) = X_n$. To reduce the impact of precise particle locations $X_n$, we propose to evaluate the ICDF in a fixed percentile grid $p_k, k=1,\dots,K$ - typically uniform between $0$ and $1$ - with $K \ll N$. This allows us to construct a coarse ICDF-to-ICDF timestepper
\begin{equation}
    \Phi_h\left(F^{-1}_t\right) = F^{-1}_{t+h}
\end{equation}
in four stages:
\begin{itemize}
    \item[1.] Interpolate $F^{-1}(t)$ on the fixed grid $p_k, k=1,\dots,K$ using a differentiable spline;
    \item[2.] Sample $\{X_n(t)\}$ by evaluating $F^{-1}(t)$ in uniform percentiles $n/N$;
    \item[3.] Propagate the particles to $\{X_n(t+h)\}$ using the particle timestepper $\phi_h$;
    \item[4.] Construct the new ICDF by sorting $\{X_n(t+h)\}$ and retaining every $K$-th particle.
\end{itemize}
A schematic of the ICDF-to-ICDF timestepper is shown in Figure~\ref{fig:icdftoicdf}.

\begin{figure}[!ht]
  \centering
  \includegraphics[width=0.8\textwidth]{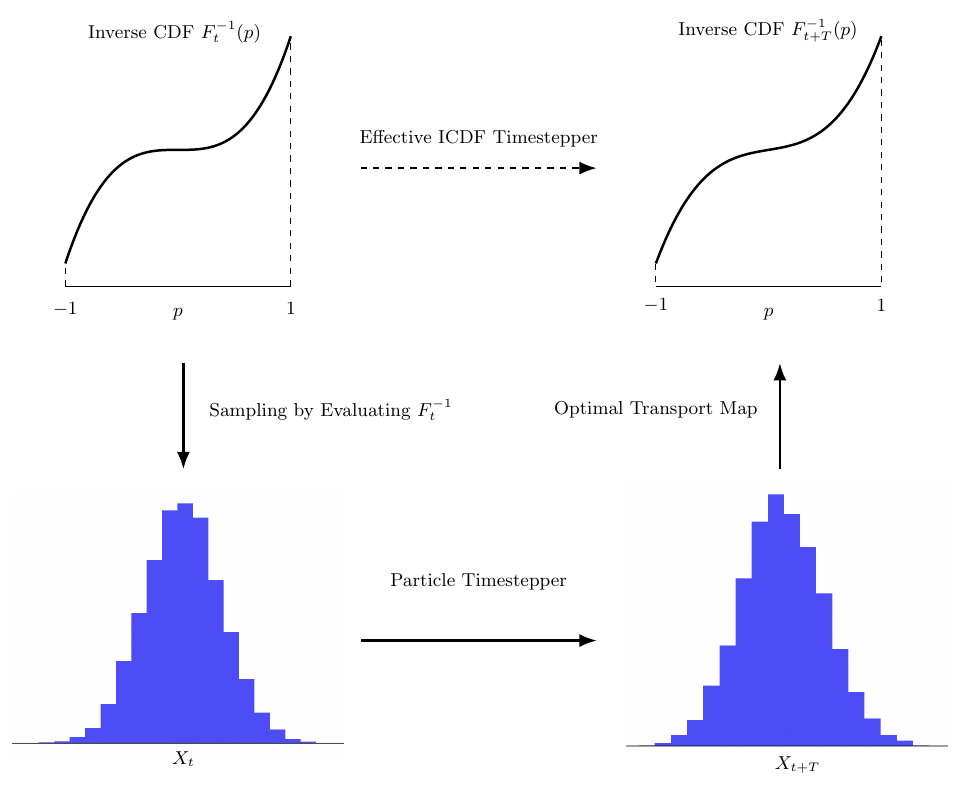}
  \caption{Schematic of the effective ICDF–to-ICDF timestepper.}
  \label{fig:icdftoicdf}
\end{figure}

By construction, both the ICDF representation and its spline approximation ensure that the timestepper maps one smooth ICDF into another. In this setting, finite-difference approximations of the directional derivative $D\Phi_h \cdot v$ no longer suffer from variance-amplifying $1/\varepsilon$ terms, since the noise inherent at the particle level has been averaged out by the smooth representation. As a result, Newton–Krylov iterations regain their expected fast convergence toward the steady-state distribution, now expressed consistently through the ICDF. The associated ICDF-based objective function is given by
\begin{equation} \label{eq:icdf_residual}
\Psi(F^{-1}) = F^{-1}- \Phi_h(F^{-1}).
\end{equation}
which is zero at steady state. Notice that the error bound in equation~\eqref{eq:nk_particle_error} remains applicable, since the ICDF-to-ICDF timestepper is ultimately constructed from the underlying particle timestepper~$\phi_h$. However, because the ICDF provides a smooth aggregate representation of the ensemble, the effective noise variance~$\sigma^2$ is substantially reduced. As a result, the stochastic error in the finite-difference approximation becomes small enough for the method to operate within the stable, “workable” Newton–Krylov regime, where second-order convergence is visible.

\subsection{Three Numerical Illustrations}
We now demonstrate the performance of the Newton–Krylov method on the ICDF-to-ICDF timestepper through three examples. These examples illustrate the accuracy, robustness, and noise tolerance of Newton--Krylov on smooth particle representations across increasingly complex systems.

\subsubsection{The Bimodal Distribution}
As a first demonstration of the Newton–Krylov method applied to the ICDF-to-ICDF timestepper, we consider a bimodal distribution defined by
\begin{equation}
    \mu(dx) = \frac{1}{Z}\exp\left(-\frac{1}{2}\left(x^2 - 1\right)^2\right).
\end{equation}
The particle timestepper is based on an Euler-Maruyama discretization of the overdamped Langevin dynamics
\begin{equation}
    dX_t = \nabla\log\mu(X_t) dt + \sqrt{2} dW_t
\end{equation}
with step size $10^{-3}$. The time-integration horizon is $h = 0.1$.

The unknowns in the ICDF-to-ICDF formulation are the values of the ICDF on a fixed percentile grid $p_k = (k - 0.5)/100$ for $1 \leq k \leq 100$. During each evaluation of the timestepper, we interpolate the ICDF using piecewise linear functions, sample $N = 10^5$ particles at equidistant percentiles, propagate them forward using the microscopic timestepper~$\phi_h$ over a time window of length~$h = 0.1$, and reconstruct the new ICDF by retaining every $1000$th particle (i.e., $N / 100$) in sorted order (i.e. after applying the optimal transport map). These retained particle locations define the updated ICDF on the fixed percentile grid at time~$h$. Jacobian–vector products are approximated by finite differences with a step size of~$\varepsilon = 10^{-2}$.

\begin{figure}[h]
    \centering
    \begin{subfigure}[b]{0.49\textwidth}
        \centering
        \includegraphics[width=\textwidth]{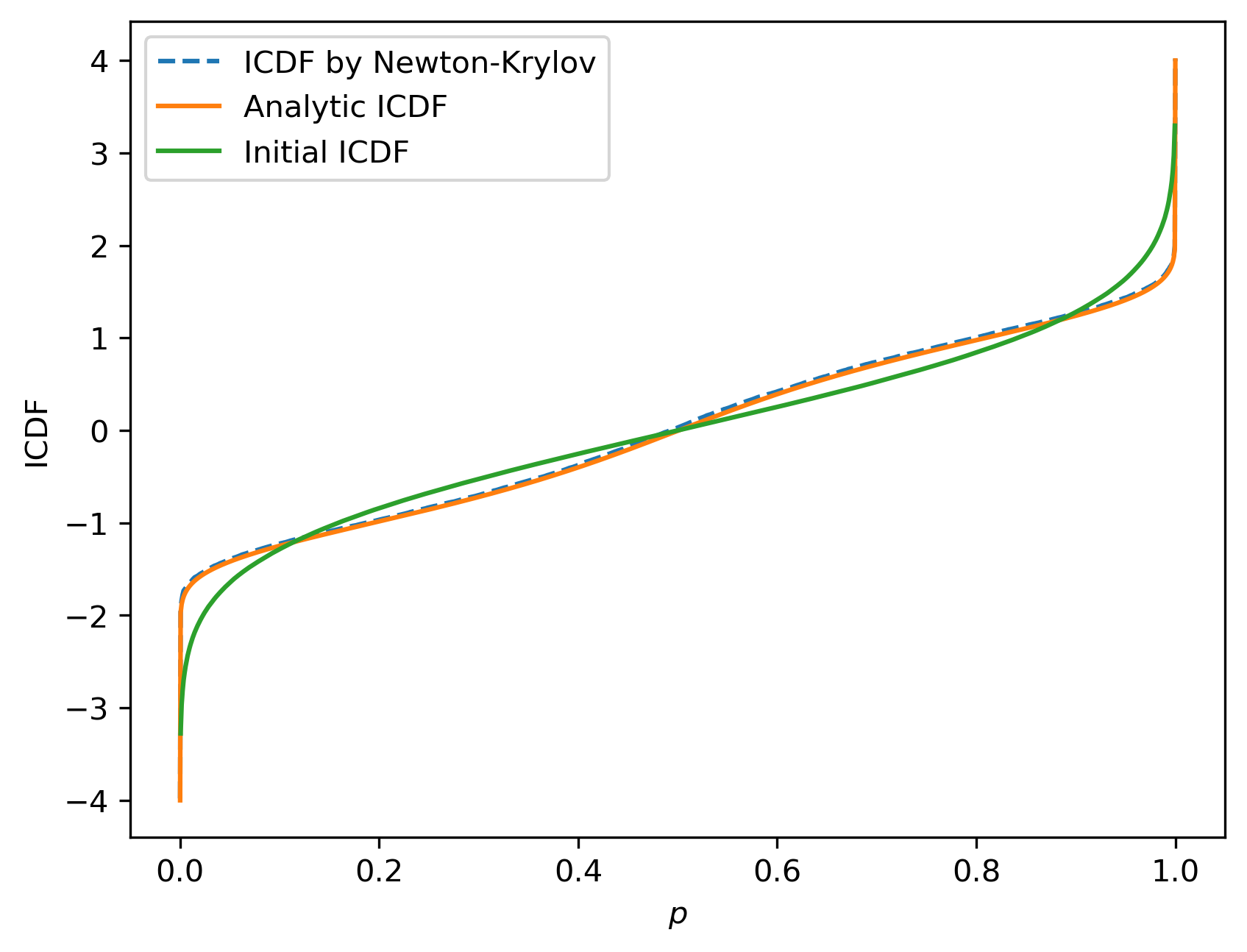}
    \end{subfigure}%
    \begin{subfigure}[b]{0.49\textwidth}
        \centering
        \includegraphics[width=\textwidth]{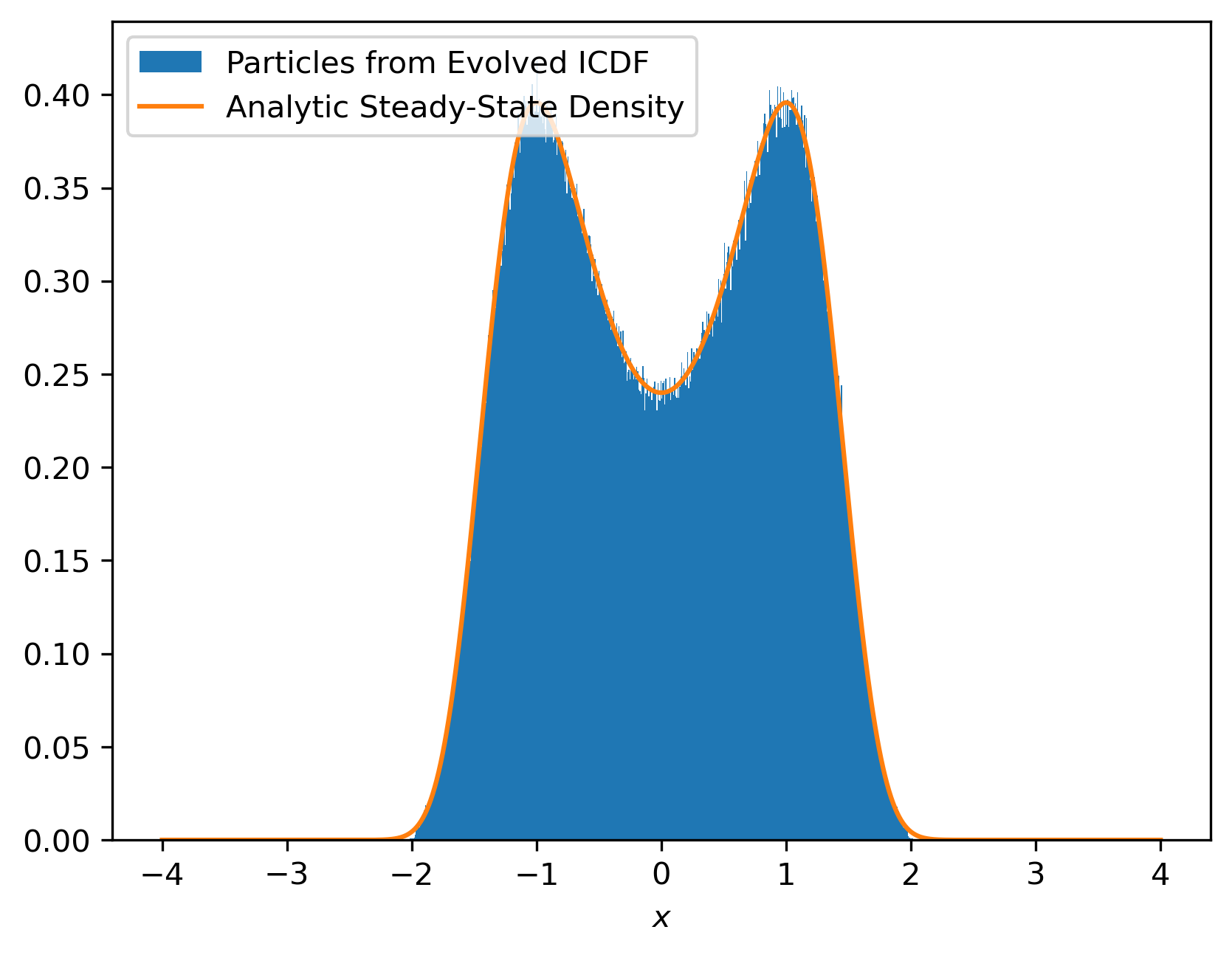}
    \end{subfigure}
    \caption{(Left) Initial ICDF (green), true steady-state (orange) and steady-state ICDF computed by Newton--Krylov (dashed blue). (Right) Particles sampled from the Newton--Krylov ICDF (blue) and corresponding steady-state bimodal density.}
    \label{fig:bimodal_icdf}
\end{figure}

The convergence behavior of Newton–Krylov on the ICDF timestepper is shown in Figure~\ref{fig:bimodal_icdf}. The left panel compares the ICDF of the initial guess—corresponding to a standard Gaussian—with the steady-state ICDF obtained after convergence. The right panel displays the histogram of $N = 10^5$ particles sampled from this steady-state ICDF, showing excellent agreement with the analytical bimodal invariant distribution. For completeness, we also report the residual norm $\norm{\Psi\left((F^{-1})^{(k)}\right)}$ per Newton–Krylov iteration~$k$, which decreases rapidly to a noise floor around $10^{-2}$. This residual level is consistent with the theoretical prediction from equation~\eqref{eq:nk_particle_error}.

The Newton--Krylov method converges to the noise floor in $11$ iterations, for a total of $63$ evaluations of the ICDF timestepper (GMRES needs a few function evaluations per global Newton iteration). This technically corresponds to $6.3$ in-simulation seconds of propagating particles to reach the steady state distribution - {\em much less than the hundreds of seconds that would be required to reach steady state using the Euler-OT timestepper}. The total in-simulation time is the most important metric for comparing optimizers.

\subsubsection{Bacterial Chemotaxis}
As a second example, we revisit the bacterial chemotaxis model introduced in section~\ref{subsubsec:w2_chemotaxis}. The ICDF is again discretized on a fixed percentile grid $p_k = (k - 0.5)/100$, for $1 \leq k \leq 100$. At each timestepper evaluation, we interpolate the ICDF using piecewise linear functions, sample $N = 10^5$ particles at equidistant percentiles, and propagate them over a time horizon of $h = 1$ second using the microscopic particle timestepper described in section~\ref{subsubsec:w2_chemotaxis}. The propagated ensemble is then projected back onto the ICDF representation by sorting the particle locations. The initial distribution is again taken to be a Gaussian with mean $5$ and standard deviation $2$, and Jacobian–vector products are again evaluated using finite differences with step size $\varepsilon = 10^{-1}$.

\begin{figure}[h]
    \centering
    \begin{subfigure}[b]{0.49\textwidth}
        \centering
        \includegraphics[width=\textwidth]{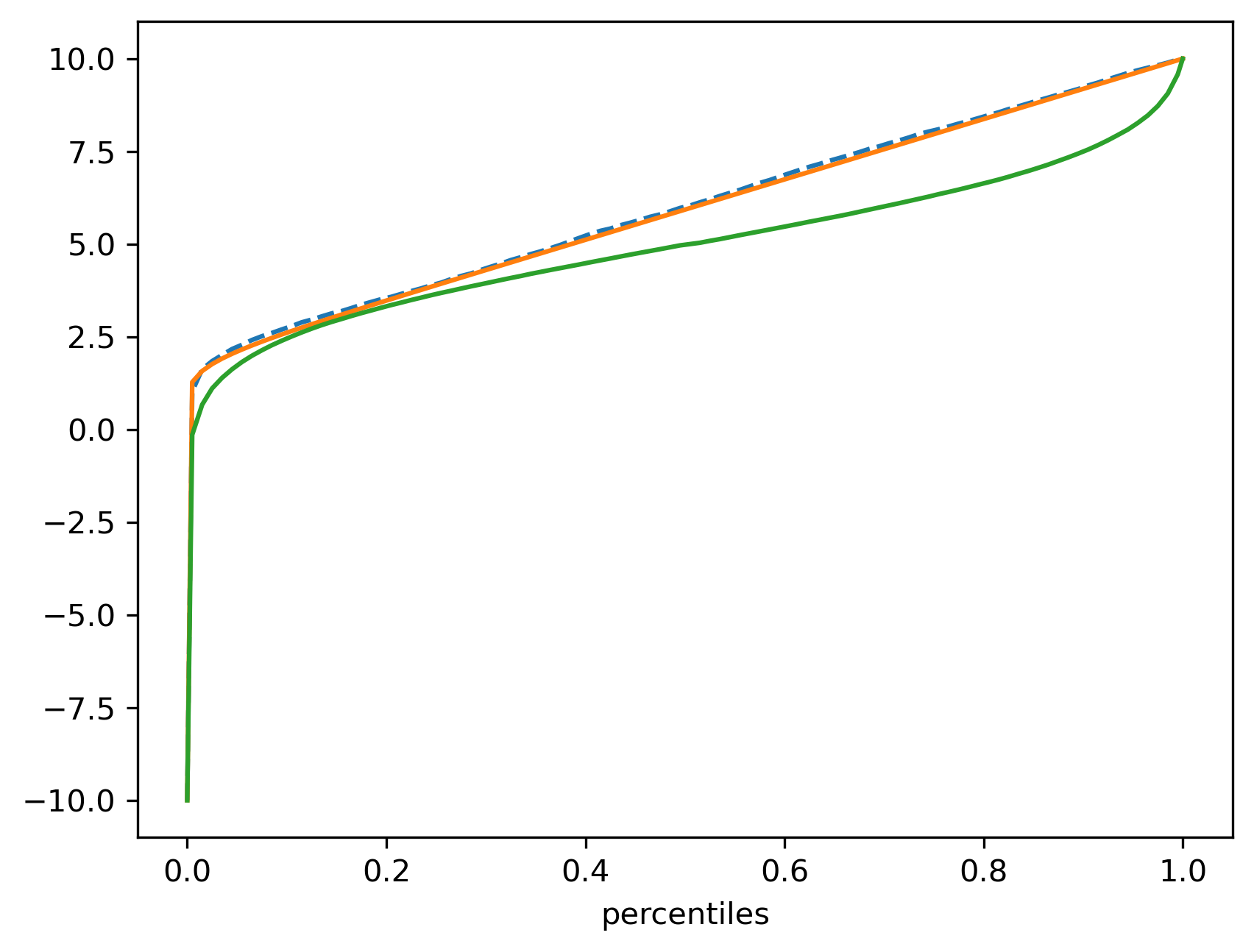}
    \end{subfigure}%
    \begin{subfigure}[b]{0.49\textwidth}
        \centering
        \includegraphics[width=\textwidth]{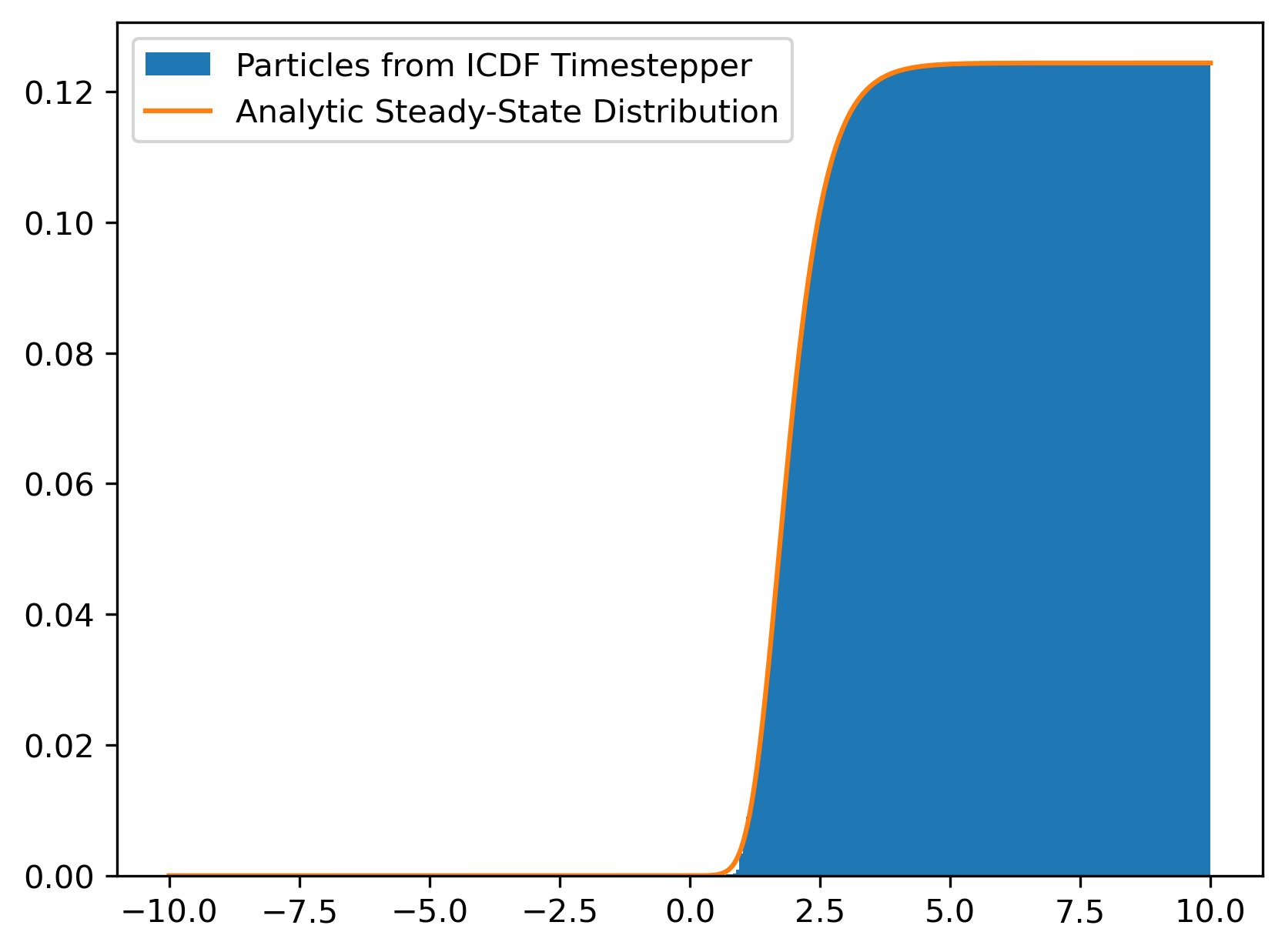}
    \end{subfigure}
    \caption{(Left) Initial ICDF(green), true steady-state (orange) and steady-state ICDF computed by Newton--Krylov (dashed blue). (Right) Particles sampled from the Newton--Krylov ICDF (blue) and corresponding steady-state density~\eqref{eq:chemotaxis_ss} in orange.}
    \label{fig:chemotaxis_icdf}
\end{figure}

Figure~\ref{fig:chemotaxis_icdf} illustrates the convergence of the Newton–Krylov method on this ICDF-to-ICDF timestepper. The left panel shows the initial ICDF (corresponding to the Gaussian guess), the true steady-state ICDF obtained from long-time simulation, and the ICDF recovered by Newton–Krylov. The latter two curves show close agreement, confirming that the method correctly identifies the stationary distribution. The right panel displays the histogram of $N = 10^5$ particles sampled from the steady-state ICDF computed by Newton–Krylov, which closely matches the invariant chemotactic density. The Newton--Krylov method is able to reach this steady-state profile after $11$ Newton iterations and $114$ objective evaluations, corresponding to $114$ seconds of in-simulation time ($h=1$ second). Regular simulation using the Euler-OT timestepper requires about $300$ seconds to reach steady state from the same initial distribution.

\subsubsection{Economic Agents}
As a final example, we apply the ICDF-to-ICDF timestepper framework to the economic agents model introduced in section~\ref{subsubsec:w2_agents}. The microscopic timestepper $\phi_h$ for this system is the McKean–Vlasov Euler scheme described in~\cite{fabiani2024task}. This model is nonlinear and non-local, and its corresponding mean-field equation~\eqref{eq:meanfield_agents} is known to admit multiple steady-state distributions.

The construction of the ICDF-to-ICDF timestepper is identical to the prior two examples. That is, we discretize the ICDF on a fixed percentile grid $p_k = (k - 0.5)/100$, for $1 \leq k \leq 100$ and $N = 10^5$ particles are sampled uniformly in percentile space. The microscopic timestepper $\phi_h$ is then applied to evolve these particles over a time horizon of $h = 1$ second. After propagation, the updated ICDF is reconstructed by sorting the particle ensemble retaining every $1000$th particle. The initial ICDF corresponds to a Gaussian distribution centered around $x = 0$ with standard deviation $0.1$. Jacobian–vector products of the timestepper are again computed using finite differences with step size $\varepsilon = 10^{-1}$.

\begin{figure}[h]
    \centering
    \begin{subfigure}[b]{0.49\textwidth}
        \centering
        \includegraphics[width=\textwidth]{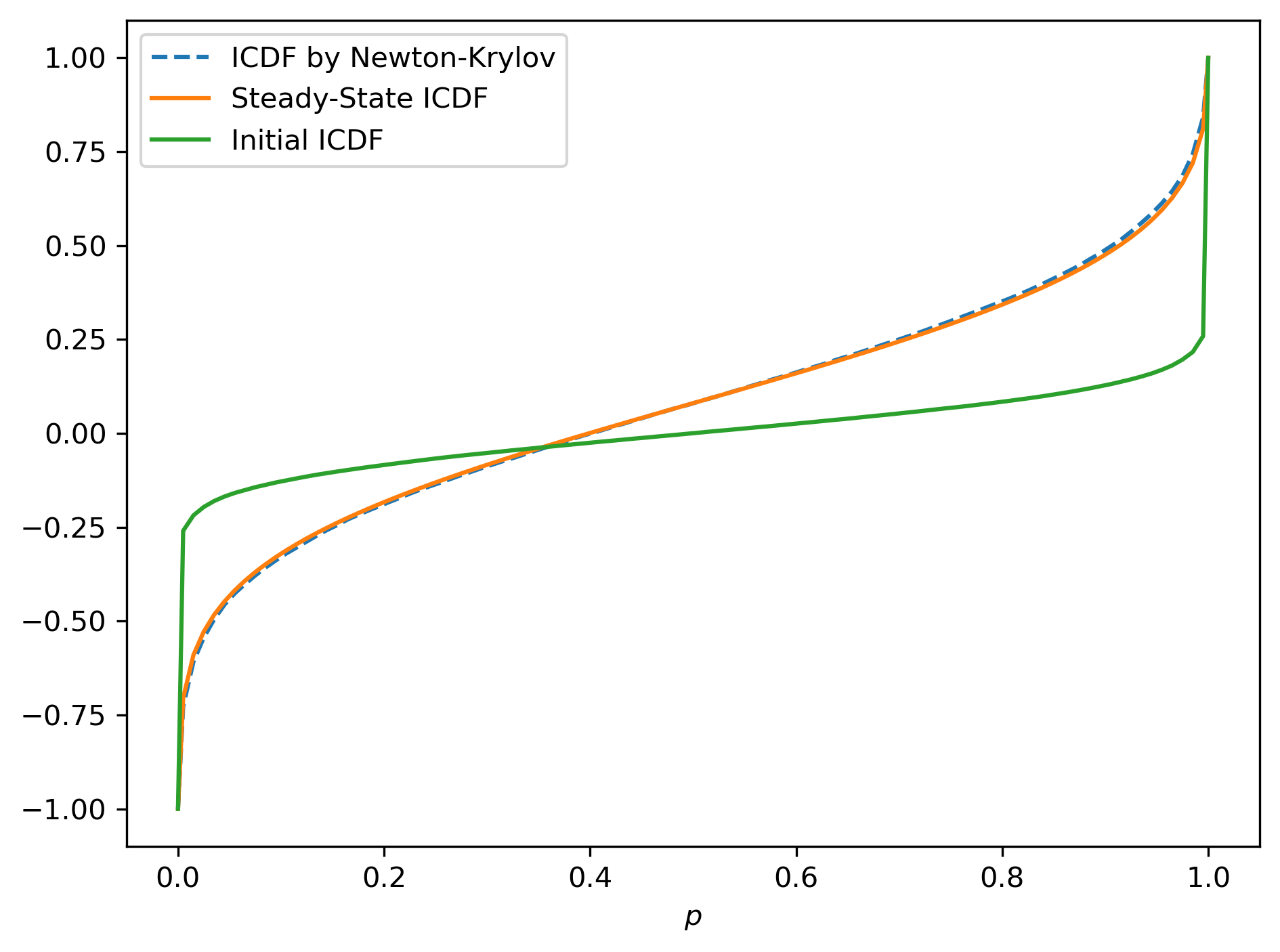}
    \end{subfigure}%
    \begin{subfigure}[b]{0.49\textwidth}
        \centering
        \includegraphics[width=\textwidth]{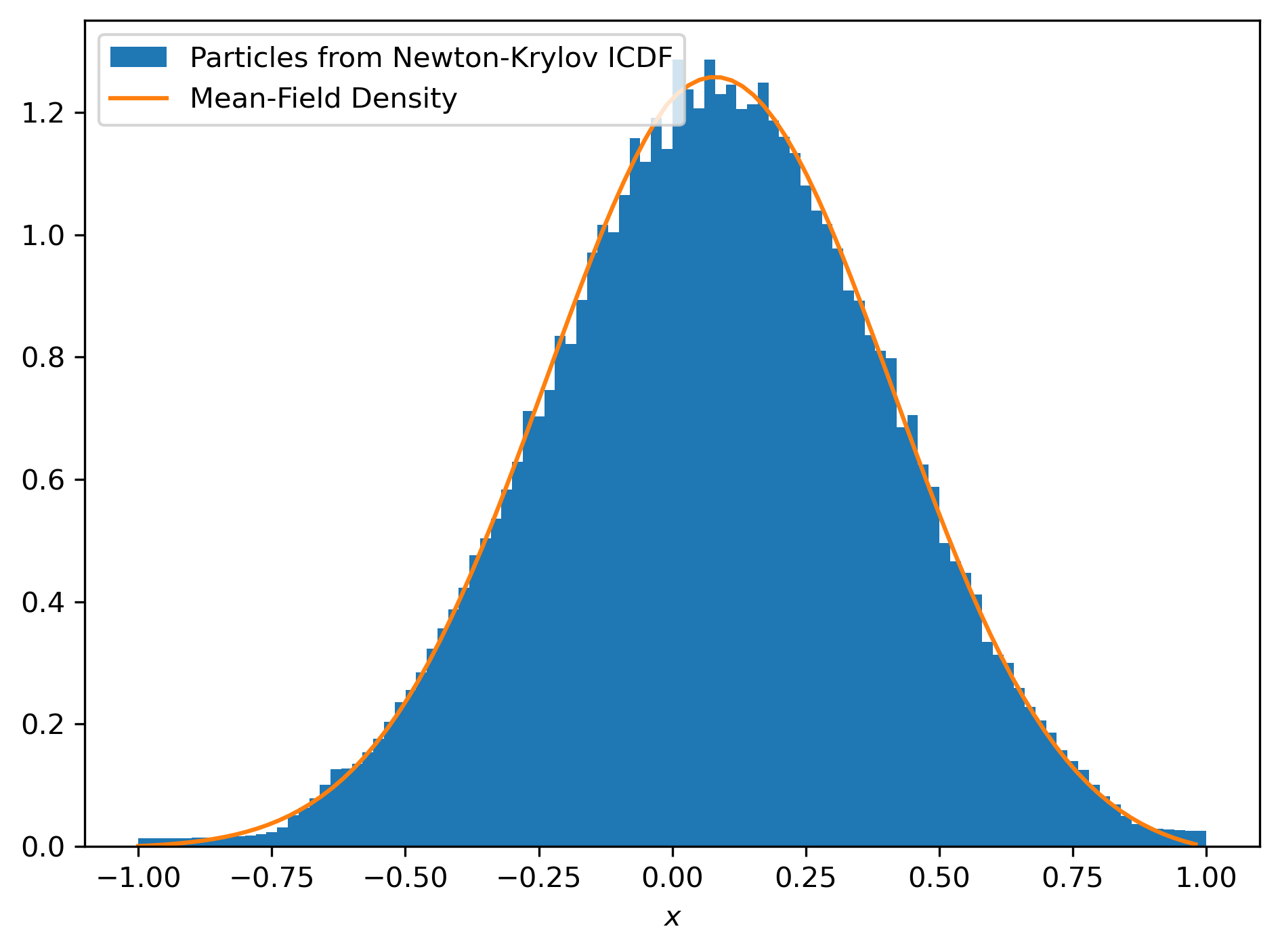}
    \end{subfigure}
    \caption{(Left) Initial ICDF (green), true steady-state (orange) and steady-state ICDF computed by Newton--Krylov (dashed blue). (Right) Particles sampled from the Newton--Krylov ICDF (blue) and steady-state density corresponding to the mean-field PDE.}
    \label{fig:economics_icdf}
\end{figure}

Figure~\ref{fig:economics_icdf} illustrates the results obtained with Newton–Krylov on this ICDF timestepper. The left panel shows the initial ICDF, the invariant ICDF obtained from long-time microscopic simulation, and the steady-state ICDF recovered by Newton–Krylov. The close agreement between the two confirms that the algorithm correctly identifies the stable steady-state distributions of the system. The right panel displays the histogram of $N = 10^5$ particles sampled from the Newton–Krylov steady-state CDF, which closely matches the invariant distribution of the mean-field PDE~\eqref{eq:meanfield_agents}. 
\begin{figure}[h]
    \centering
    \includegraphics[width=0.6\textwidth]{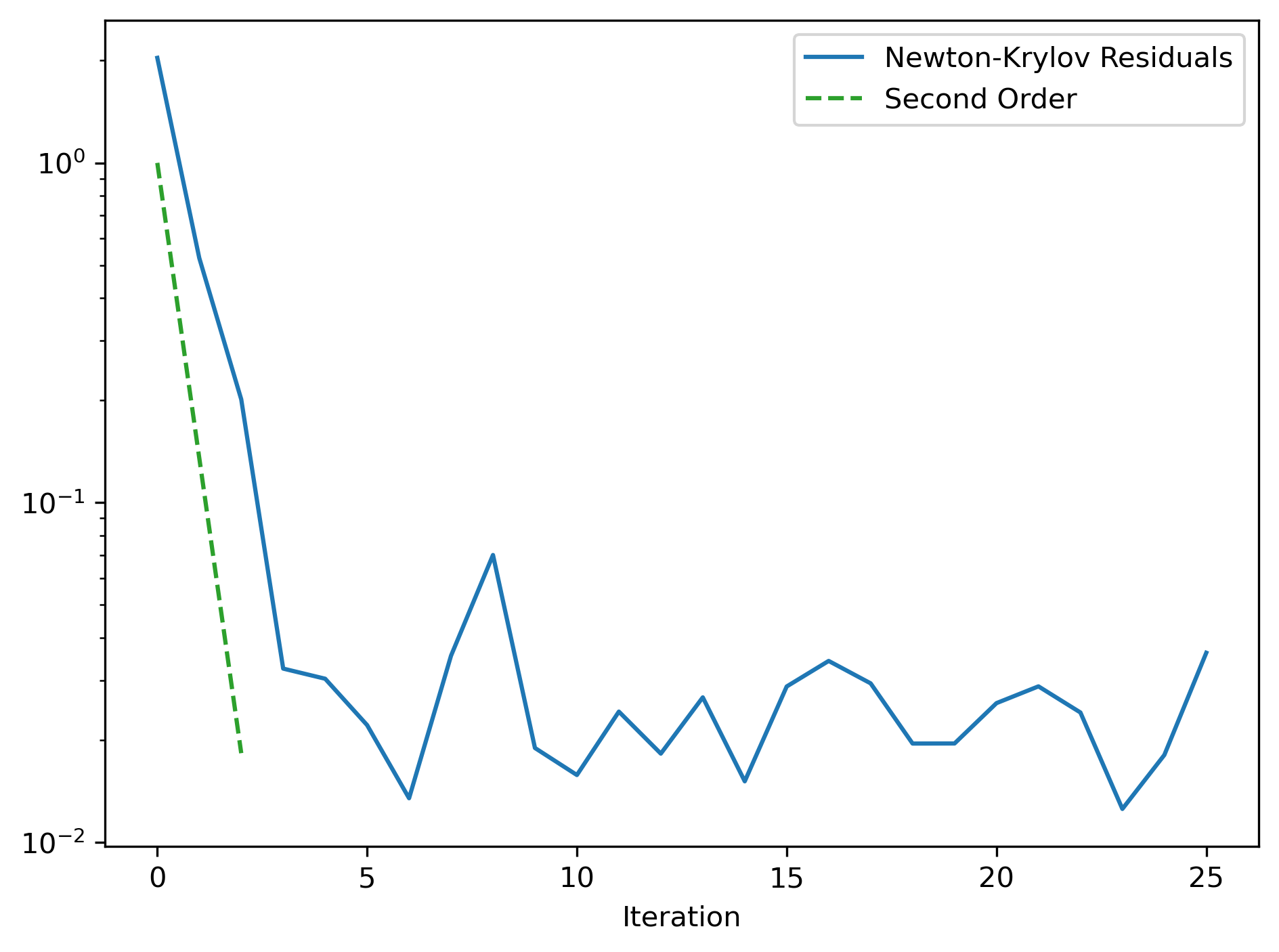}
    \caption{Newton--Krylov ICDF residual $\Psi\left((F^{-1})^{(k)}\right)$ per iteration for the economics agents model (blue); Second-order convergence rate (green). }
    \label{fig:economics_icdf_losses}
\end{figure}
For completeness we also show the Newton--Krylov error as a function of the iteration number in Figure~\ref{fig:economics_icdf_losses}. We see that the error first decreases quadratically as is typical for Newton-like schemes, before settling down in the noise regime.

\section{Extending Smooth Representations to Multiple Dimensions} \label{sec:higher_dimensions}
In one dimension, the ICDF approach allows us to handle the issue presented by noisy particles and still develop an appropriate Newton--Krylov method.
Extending this idea beyond one dimension requires additional care. There is not an analog of the ICDF that appropriately captures the optimal transport maps. In particular, an inverse cumulative distribution function cannot be defined in a straightforward manner for $d \ge 2$. The reason is that the CDF is not injective: for a given probability level $p \in [0,1]$, the level set
\begin{equation}
    \left\{(x,y) \in \mathbb{R}^2 \mid F_{X,Y}(x,y)=p\right\}
\end{equation}
typically forms a one-dimensional manifold (a curve) rather than a single point. Hence, the mapping $\left(F_{X,Y}\right)^{-1}(p)$ is not a well-defined function from $[0,1] \to \mathbb{R}^2$. Consequently, the notion of “evaluating the ICDF” becomes ambiguous in higher dimensions. It follows that the one-dimensional ICDF-based timestepper illustrated in Figure \ref{fig:icdftoicdf} cannot be extended to multivariate distributions in a direct manner. Even if a generalized notion of inverse CDF were introduced (e.g., by parameterizing level sets), its numerical evaluation would be computationally inefficient, since sampling points uniformly on such two-dimensional level sets is considerably more expensive than evaluating the scalar inverse $F^{-1}(p)$ in one dimension.

One must instead construct alternative smooth representations that retain the essential link to Optimal Transport theory. An option is to work with the two-dimensional cumulative distribution function (CDF), which generalizes the one-dimensional case and retains links to optimal transport. An alternative is the sliced Wasserstein distance. By projecting the distribution onto many one-dimensional directions, each projection yields an ICDF that is well defined, and these are aggregated to approximate the full Wasserstein geometry. Therefore, both the CDF and the sliced Wasserstein framework act as higher-dimensional analogues of the one-dimensional ICDF, providing the smoothness needed to construct well-behaved timesteppers and to restore the accelerated convergence of Newton–Krylov methods in the multidimensional setting.

It is worth emphasizing, though, that the reason we need an alternative approach is purely for computational feasibility. If we had an oracle that could instantaneously compute the optimal transport map between any two distributions of arbitrary size, our original approach would still work in multiple dimensions.
Indeed, if we were able to compute the associated vector field in the case of continuous distributions, then there would not be any noise, and we could run the Newton--Krylov method exactly as expected.
However, due to the noisy nature of the timestepper, our Newton--Krylov approach needs more particles in the simulation than would be reasonable to pass to an explicit OT solver.

This computational issue only appears in the setting of multiple ambient dimensions precisely because optimal transport maps are extremely cheap to compute in one dimension.
For multiple dimensions, computing an optimal transport map requires solving a linear program which is no longer practical for the number of particles we would need in order for the desired effects to become observable.
As we have discussed, the signal-to-noise ratio in the computation of the Jacobians is extremely poor for small numbers of particles, and in order to make the ratio high enough for our Newton--Krylov method to work, we need to make the number of particles so high that solving the linear program becomes computationally infeasible.
Additionally, the statistical rate of convergence of discrete optimal transport maps depends poorly on the dimension (typically \(n^{-1/d}\), where \(n\) is the number of points and \(d\) is the ambient dimension), meaning the number of particles needed also scales with the dimension.
This creates a bad tradeoff with computational complexity that cannot be resolved without unreasonable computational power.
Thus, the reason we need additional alternative approaches in multiple dimensions is the practical computational ability, not some fundamental flaw with the approach of running Newton--Krylov on the OT velocity fields.

We discuss two such alternatives in detail here: the multidimensional cumulative distribution function and the sliced-Wasserstein distance. We discuss the former idea in section~\ref{subsec:cdf2d} and the latter in section~\ref{subsec:sw2}.

\subsection{The CDF-to-CDF Timestepper} \label{subsec:cdf2d}
A natural choice is to work directly with the two-dimensional cumulative distribution function $F_{X,Y}(x,y)$, which provides a continuous and differentiable description of the underlying particle ensemble. The notion of a cumulative distribution function (CDF) extends naturally to multiple dimensions. For clarity of exposition, we focus on the bivariate case, although all subsequent algorithms and results generalize to arbitrary number of dimensions. Let $(X,Y)$ be a random vector with joint distribution. Its cumulative distribution function is defined by
\begin{equation}
 F_{X,Y}(x,y) = P(X\leq x, Y\leq y) = \int_{-\infty}^x \int_{-\infty}^y \mu(u,v) dudv.
\end{equation}

Sampling from a two-dimensional cumulative distribution function can be performed through its marginal and conditional components. Following the approach described by~\cite{zou2005equation,zou2006equation}, the joint CDF $F_{X,Y}(x,y)$ can be decomposed into the marginal CDF $F_X(x)$ of one variable and the conditional CDF $F_{Y|x}(y|x)$ of the other. This decomposition allows sampling to proceed through a sequence of one-dimensional operations: one first draws a sample $x$ from the marginal cumulative distribution $F_X(x)$, and subsequently samples $y$ from the conditional cumulative distribution $F_{Y|X}(y)$. In this way, multidimensional sampling reduces to iterated evaluations of smooth one-dimensional ICDFs, which are mathematically well-defined.

The two-dimensional sampling algorithm works as follows. First we construct the marginal CDF of the $x$-coordinates, given by
\begin{equation} \label{eq:marginal_cdf}
    F_X(x) = P(X \leq x) = \int_{-\infty}^x \int_{-\infty}^{\infty} \mu(u,v) dudv = F_{X,Y}(x,\infty).
\end{equation}
This one-dimensional CDF can be easily inverted to generate samples $x_1, x_2, \dots, x_n$ at equidistant percentiles $p_i = (i-0.5)/n$. Next, for each $x$-sample $x_i$ we construct the one-dimensional CDF of the conditional distribution $\mu(y|x_i)$. It can be shown~\cite{papoulis1965random,probability1995measure} that this conditional cumulative distribution function is given by the formula
\begin{equation} \label{eq:conditional_cdf}
 F_{Y|X_i}(y) = \frac{\partial_x F_{X,Y}(X_i,Y)}{\mu_X(X_i)}.
\end{equation}
Given any smooth representation of $F_{X,Y}(x,y)$ this partial derivative is readily available, and sampling the $y$-coordinates $X_i$ can also be achieved by inverting $F_{Y|X_i}(y)$ and evaluating it in equidistant percentiles $q_j = (j-0.5)/m$. The end product of this staged sampling algorithm are $N$ particles $(X_n, Y_n)_{n=1}^N$. After propagating the particles through the timestepper
\begin{equation}
    (\tilde{X}_n,\tilde{Y}_n) = \phi_h(X_n,Y_n),
\end{equation}
we reconstruct the updated two-dimensional CDF by computing, at each grid node $(x_i, y_j)$, the fraction of particles $(\tilde{X}_n, \tilde{Y}_n )_{n=1}^N$ located in the lower-left quadrant relative to that point:
\begin{equation}
    F_{X,Y}(x_i,y_j) = \#\left\{ n \ | \ \tilde{X}_n \leq x_i \ \text{and} \ \tilde{Y}_n \leq y_j \right\} / N
\end{equation}
Starting from the CDF $F_{X,Y}^t$ at time $t$, the three steps 
\begin{itemize}
\item[1.] Sampling $(X_n,Y_n)$ from marginal and conditional CDFs; 
\item[2.] Forward particle propagation using the particle timestepper $(\tilde{X}_n, \tilde{Y}_n) = \phi_h(X_n, Y_n)$;
\item[3.] Restriction to the 2D CDF by counting particles in the lower right quadrant of each grid point;
\end{itemize}
define a consistent CDF-to-CDF timestepper
\begin{equation}
    \Phi_h\left(F_{X,Y}^t\right) = F_{X,Y}^{t+h}.
\end{equation}
See Figure~\ref{fig:cdftocdf} for a schematic of this timestepper. The associated residual map reads
\begin{equation} \label{eq:2dcdf_psi}
\Psi\left(F_{X,Y}^t\right) = F_{X,Y}^t - \Phi_h(F_{X,Y}^t),
\end{equation}
which is identical to $0$ at every grid point $(x_i,y_j)$ in steady state. 

We note that constructing the two-dimensional CDF of a particle ensemble and re-sampling it through the marginal and conditional one-dimensional CDFs, as we outlined above, yields a monotone triangular map that is close to the OT solution under many practical conditions. In fact, it corresponds to the Knothe–Rosenblatt rearrangement~\cite{rosenblatt1952remarks,knothe1957contributions,santambrogio2015optimal,peyre2019computational}, a sequential, measure-preserving transformation that orders variables one at a time through their conditional distributions. While the Knothe–Rosenblatt map does not minimize the standard quadratic OT cost (and hence is not itself an optimal transport map in the usual sense), it can be expressed as a limit of optimal transport maps for different cost functions \cite{carlier2010knothe}.
In particular, it shares key structural properties and can be viewed as a computationally efficient approximation to the optimal transport map in high-dimensional sampling contexts. We do not require the exact optimal transport map in our Newton–Krylov calculations; any fast and reasonably accurate alternative that preserves the steady state is sufficient. The Knothe–Rosenblatt rearrangement offers a good alternative that is computationally efficient and preserves the steady state.

\begin{figure}[!ht]
  \centering
  \includegraphics[width=0.8\textwidth]{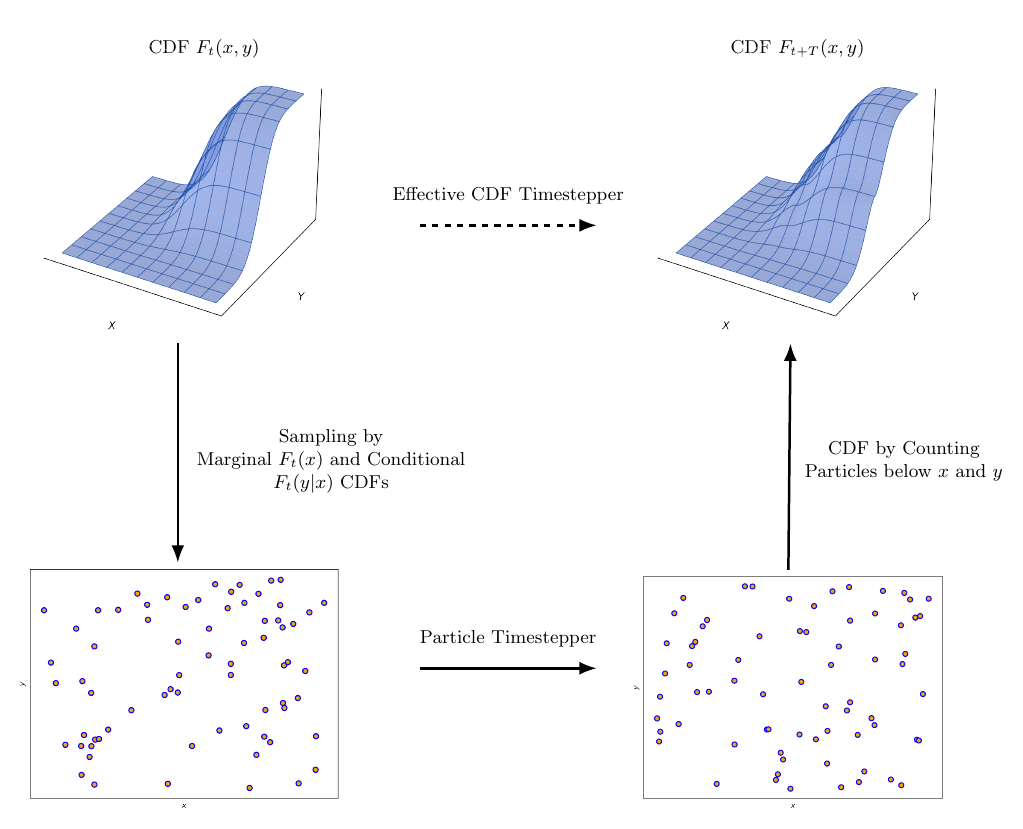}
  \caption{Schematic of the effective two-dimensional CDF–to-CDF timestepper.}
  \label{fig:cdftocdf}
\end{figure}

\paragraph{Numerical Example}
Consider the two-dimensional half-moon distribution~\eqref{eq:halfmoonpotential} as a representative example of our approach. We evaluate the empirical two-dimensional cumulative distribution function (CDF) on a uniform grid $(x_i, y_j)_{i,j}$ of $n_X = n_Y = 100$ points, equally spaced between $-4$ and $4$. The CDF-to-CDF timestepper is constructed using $N = 10^4$ particles, and sampling is performed by inverting the marginal and conditional one-dimensional cumulative distribution functions. To increase regularity, we interpolate the two-dimensional CDF with a piecewise-linear spline and solve for the corresponding percentiles. Specifically, for every percentile pair $(p_i, q_j)$, we determine the coordinates $(X_i, Y_j|X_i)$ such that
\begin{equation}
    F_X(X_i) = p_i, \ F_{Y|X_i}(Y_j) = q_j.
\end{equation}
This inversion can be implemented efficiently by vectorizing the evaluation of the marginal and conditional CDFs. Because the two-dimensional CDF is represented as a smooth spline, its partial derivatives—such as those appearing in~\eqref{eq:conditional_cdf}—are spline functions as well and need to be computed only once, independent of $p_i, q_j, X_i$ or $Y_j|X_i$.

We next apply the Newton--Krylov scheme to find the steady-state distribution of the CDF-to-CDF timestepper. The timestepper is based on an Euler-Maruyama discretization of the dynamics~\eqref{eq:langevin_halfmoon} with time step $10^{-3}$. We integrate this dynamics up to time $h = 1$ second. The initial condition for the Newton--Krylov method is the standard bivariate Gaussian distribution (see the left of Figure~\ref{fig:nk_2dcdf}). 
\begin{figure}[h]
    \centering
    \begin{subfigure}[b]{0.49\textwidth}
        \centering
        \includegraphics[width=0.852\textwidth]{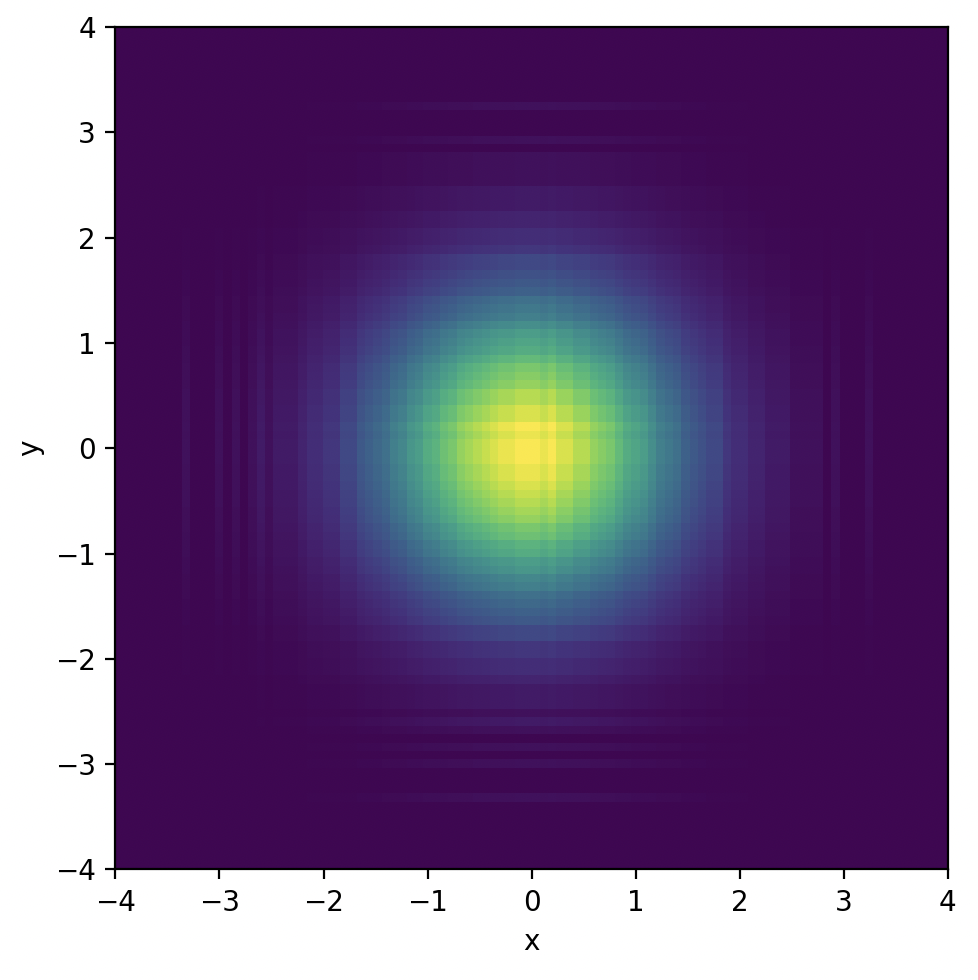}
    \end{subfigure}%
    \begin{subfigure}[b]{0.49\textwidth}
        \centering
        \includegraphics[width=\textwidth]{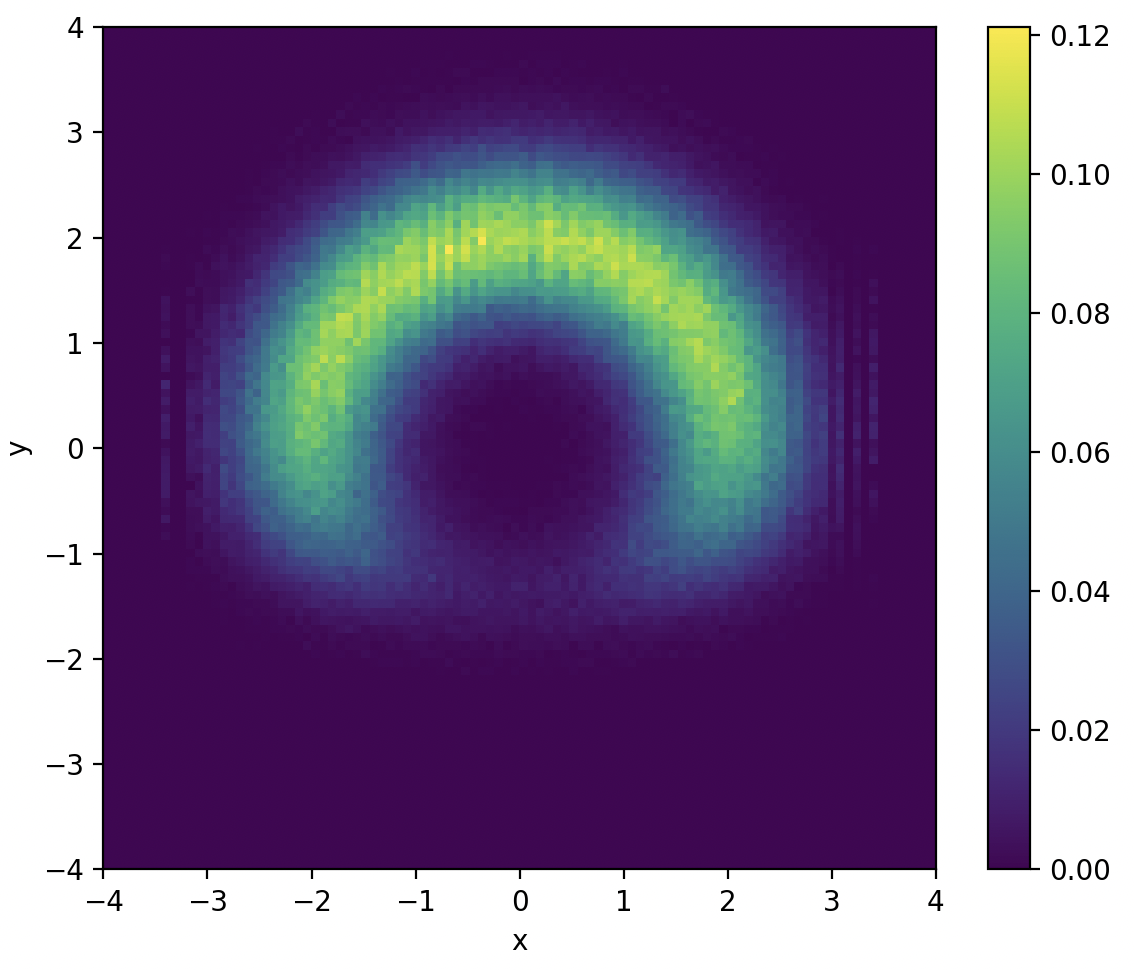}
    \end{subfigure}
    \caption{Histogram heatmaps of particle density: (left) initial Gaussian condition; (right) distribution after Newton–Krylov optimization using the 2D CDF smooth representation.}
    \label{fig:nk_2dcdf}
\end{figure}
The optimized distribution obtained by Newton--Krylov is displayed on the right of Figure~\ref{fig:nk_2dcdf}. One can see that Newton--Krylov can recover the true steady-state distribution even though the initial distribution is quite far from equilibrium. Figure~\ref{fig:nk_2dcdf_loss} shows how the `loss` $\norm{\Psi\left(F_{X,Y}\right)}$ decreases steadily per nonlinear iteration and settles down to the noise level, a clear signal of convergence.

\begin{figure}[h]
    \centering
    \includegraphics[width=0.5\textwidth]{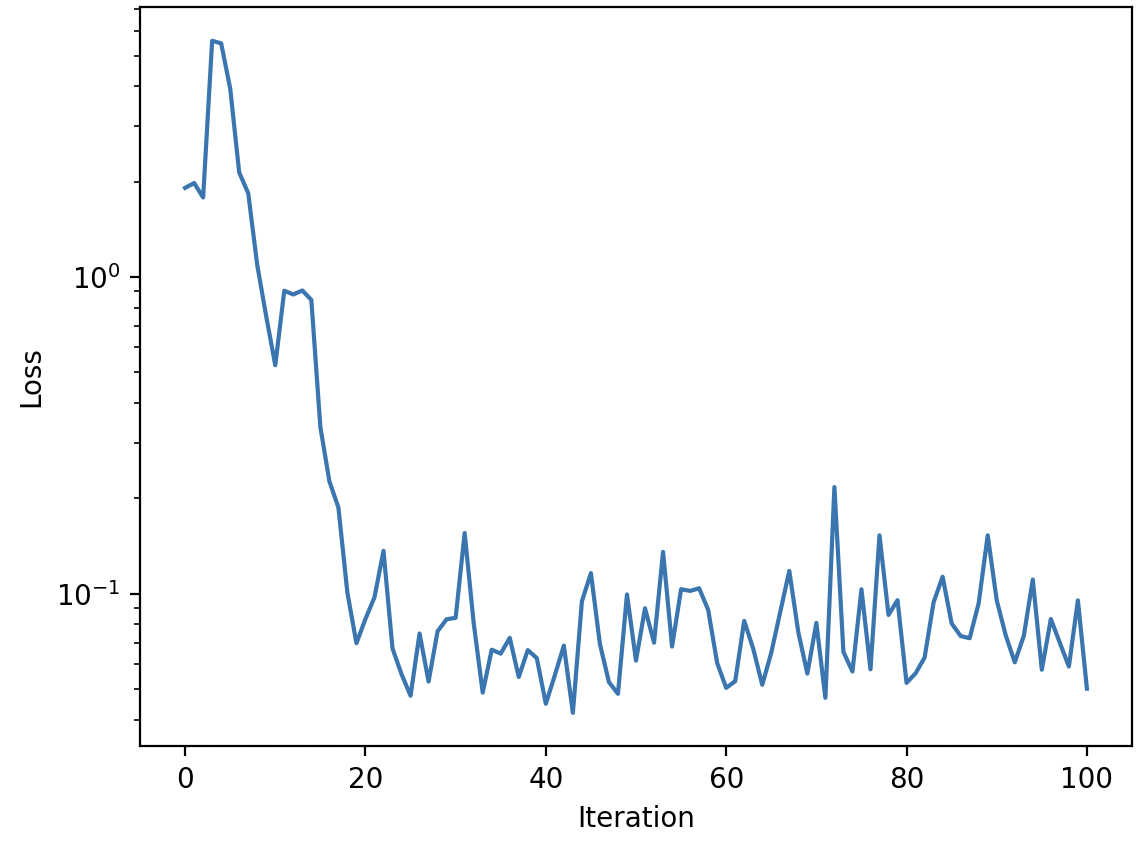}
    \caption{Newton--Krylov loss $\norm{\Psi\left(F^{(k)}_{X,Y}\right)}$ per iteration $k$. The loss decreases up until a noise level induced by the finite number of particles.}
    \label{fig:nk_2dcdf_loss}
\end{figure}

\subsection{The Sliced Wasserstein Timestepper} \label{subsec:sw2}
For completeness, we also consider an alternative smooth representation of multidimensional particle systems based on the sliced Wasserstein distance (SWD). The key idea is to project probability distributions in $\mathbb{R}^d$ onto one-dimensional subspaces defined by unit vectors $\theta \in \mathbb{S}^{d-1}$ at percentiles of the angular CDF, and to compute the Wasserstein distance between the resulting one-dimensional projected measures. Formally, for two distributions $\mu$ and $\nu$,
\begin{equation}
  SW_d^2(\mu,\nu) = \int_{\mathbb{S}^{d-1}}
  W_{d-1}^2\left(P_\theta \# \mu, P_\theta \# \nu\right)\,\mathrm{d}\theta,
\end{equation}
where $P_\theta(x)=\langle x,\theta\rangle$ is the projection onto the line spanned by~$\theta$ and $P_\theta \# \mu$ denotes the pushforward of $\mu$ by $P_\theta$.
In two dimensions, we approximate the integral with a finite set of angles $\{\theta_i\}_{i=1}^{N_\theta}$:
\begin{equation}
  SW_2^2(\mu,\nu)
  \;\approx\;
  \frac{1}{N_\theta}\sum_{i=1}^{N_\theta}
  W_2^2\!\big(P_{\theta_i}\#\mu,\; P_{\theta_i}\#\nu\big)
  \;=\;
  \frac{1}{N_\theta}\sum_{i=1}^{N_\theta}
  \int \!\big|F_{P_{\theta_i}\mu}^{-1}(r)-F_{P_{\theta_i}\nu}^{-1}(r)\big|^2\,\mathrm{d}r,
\end{equation}
where $F_{P_{\theta_i}\mu}$ is the one-dimensional CDF of the projected measure $P_{\theta_i}\#\mu$.
We note that the sliced Wasserstein distance lower bounds the ordinary Wasserstein distance, and it also gives an upper bound up to a constant that depends on the ambient dimension \cite{bonneel2015sliced}.
Thus, the sliced Wasserstein distance is indeed a reasonable approximation, preservers the steady state, and we can still use it effectively build a similar timestepper.

The sliced Wasserstein idea provides an alternative to build smooth timesteppers from an underlying particle timestepper. We retain a rectangular grid representation of the two-dimensional CDF, but sampling proceeds through directional projections. To achieve a consistent sampling, we need to first sample directions $\theta_i$ from the marginal angular CDF $F_\Theta$ and, for each $\theta_i$, generate radii $r_j|\theta_i$ from the conditional CDF $F_{R|\theta_i}(r)$. However, unlike the Cartesian marginal and conditional CDFs, there are no explicit expressions analogous to equations~\eqref{eq:marginal_cdf} and~\eqref{eq:conditional_cdf} for the marginal angular and conditional radial CDFs. We need to first explicitly construct the two-dimensional density
\begin{equation}
    \mu(x,y) = \nabla F_{X,Y}(x,y)
\end{equation}
which can be obtained efficiently from a spline representation of $F_{X,Y}$. Then the marginal angular CDF is given by
\begin{equation} \label{eq:marginal_angular}
    F_{\Theta}(\theta)
    = \int_{0}^{\theta} \int_{0}^{\infty}
        \mu \left(r \cos\phi,\, r \sin\phi\right)\,
        r\, dr\, d\phi.
\end{equation}
and the conditional radial CDF for any angle $\theta$ reads
\begin{equation} \label{eq:conditional_radial}
    F_{R|\theta}(r) = \int_{0}^{r} \mu(s \cos\theta, s\sin\theta)\ s\ ds.
\end{equation}
Sampling the 2D CDF consistently in a sliced Wasserstein-inspired way can then be achieved through
\begin{enumerate}
  \item Generating a set of angles $\{\theta_i\}_{i=1}^{N_\Theta}\subset[0,2\pi)$ by inverting the marginal angular CDF~\eqref{eq:marginal_angular} in fixed percentiles $p_i = (i - 0.5) / N_\Theta$;
  \item For each $\theta_i$, constructing the radial conditional CDF $F_{R|\theta_i}(r)$ using the projection~\eqref{eq:conditional_radial};
  \item Inverting each radial conditional CDF in fixed percentiles $\{r_j\}_{j=1}^{N_r}\subset(0,1)$ through a bisection-like scheme
        \[
          \rho_{ij} \;=\; F_{R | \theta_i}^{-1}(r_j).
        \];
  \item Mapping these 1D samples back to $\mathbb{R}^2$ via the inverse projection
        \[
          X_n=\rho_{ij}\cos\theta_i, \qquad Y_n=\rho_{ij}\sin\theta_i,
        \]
        and collecting all $(X_n,Y_n)$ across $i,j$.
\end{enumerate}
After sampling we proceed as with the CDF-to-CDF timestepper by propagating the samples through the particle timestepper to obtain new particles $(\tilde{X}_n,\tilde{Y}_n)$ and restricting back to the CDF by counting the number of $(\tilde{X}_n,\tilde{Y}_n)$ to the lower left of each CDF grid point $(x_i,y_j)$. Together, this sliced representation captures much of the transport geometry while remaining computationally light, providing a scalable surrogate for the full multidimensional Wasserstein map within our Newton--Krylov framework. The complete sliced Wasserstein-to-sliced Wasserstein timestepper is shown schematically in Figure~\ref{fig:swtosw}. 

\begin{figure}[!ht]
  \centering
  \includegraphics[width=0.8\textwidth]{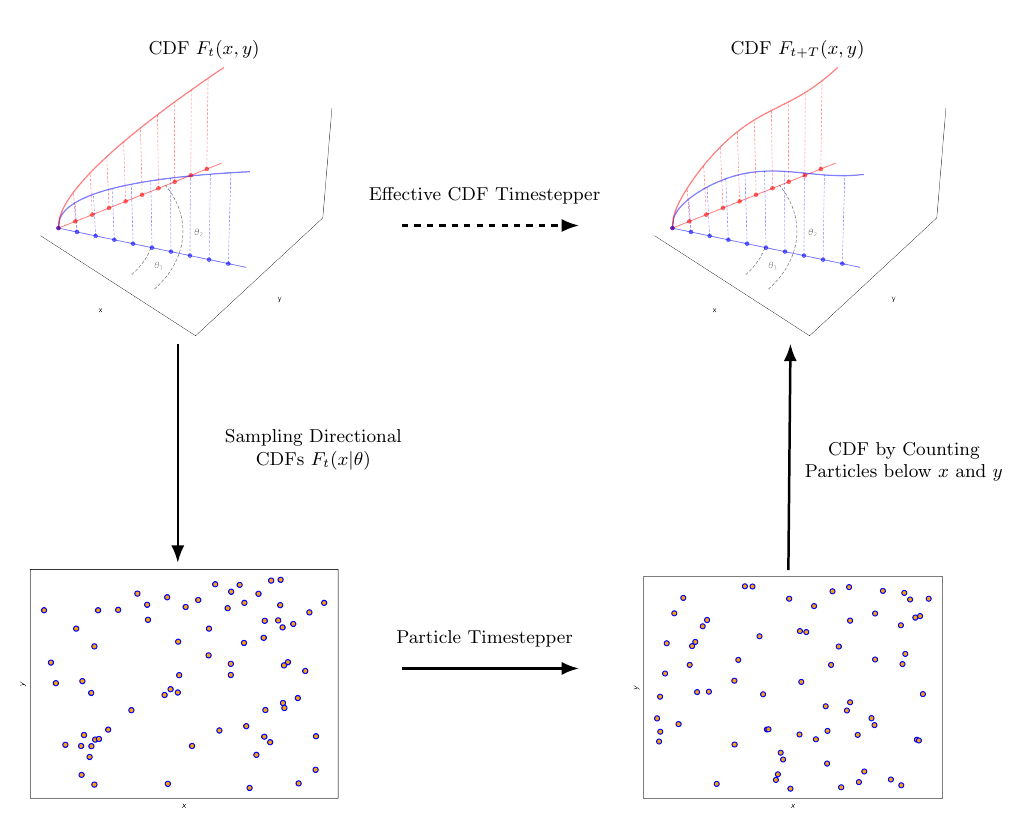}
  \caption{Schematic of the effective sliced Wasserstein to sliced Wasserstein timestepper.}
  \label{fig:swtosw}
\end{figure}

Finally we also show how the Newton--Krylov method on the sliced-Wasserstein representation of particles converges to the steady-state distribution of the half-moon potential. The initial condition and timestepper parameters are the same as in section~\ref{subsec:cdf2d}, and sampling is done by first generating $N_\Theta=100$ angular samples and then $N_R=100$ radial samples for each angle. The Newton--Krylov loss per iteration is shown in the left subfigure of Figure~\ref{fig:nk_sw2} and the resulting steady-state distribution is shown on the right.

\begin{figure}[h]
    \centering
    \begin{subfigure}[b]{0.5\textwidth}
        \centering
        \includegraphics[width=\textwidth]{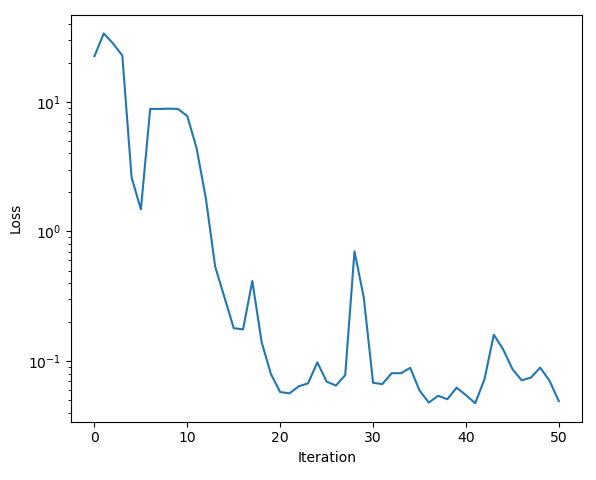}
    \end{subfigure}%
    \begin{subfigure}[b]{0.5\textwidth}
        \centering
        \includegraphics[width=\textwidth]{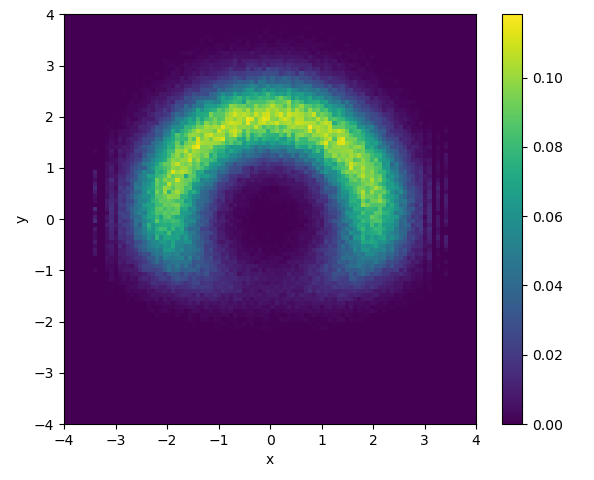}
    \end{subfigure}
    \caption{(left) Loss $\norm{\Psi(F_{X,Y}^{(k)}}$ per iteration; (right) Histogram heatmap of resulting steady-state particle density of the Newton--Krylov method applied to the Sliced Wasserstein representation.}
    \label{fig:nk_sw2}
\end{figure}

As shown, there is no significant difference in either the convergence rate or the resulting steady state between the sliced Wasserstein and the direct CDF representations. This outcome is expected, since both formulations are mathematically equivalent. The key insight is that employing smooth representations enables the use of higher-order optimization schemes, while the specific choice of representation is comparatively unimportant. The decisive factor is the computational efficiency of sampling.

\section{Discussion and Outlook} \label{sec:conclusion}
We have presented a unified, matrix-free framework for computing steady-state distributions of (stochastic) particle timesteppers. The key idea is to reformulate a steady state in the language of optimal transport. On the deterministic side, we revisited the residual formulation $\psi(u) = u - \phi_h(u)$ and showed how the Newton–Krylov method efficiently recovers steady states of the Fokker-Planck equation. We also derived a practical criterion for selecting the integration horizon $h$ from the spectral gap. On the stochastic side, we introduced a first-order Adam-Wasserstein method for calculating steady states directly on the particle level. We also clarified why a naive particle-level extension fails to higher-order optimizers such as Newton--Krylov. Stochastic noise breaks the one-to-one correspondence between ensembles, and Jacobian–vector products acquire variance scaling as $1/\varepsilon$. Our analysis makes this bias–variance trade-off explicit and shows that stable convergence can be achieved when the finite-difference step size is chosen within a noise-dependent range.

To address this limitation, we introduced smooth distributional timesteppers—first in one dimension through the ICDF-to-ICDF map, and then in multiple dimensions through CDF-to-CDF and sliced Wasserstein formulations. These representations aggregate microscopic variability into smooth macroscopic objects on which Newton–Krylov regains its fast, second-order convergence. Numerical results confirm that these smooth timesteppers yield comparable steady states with markedly reduced stochastic fluctuations, enabling accurate steady-state computations even in noisy particle systems. The central message is that smoothness in representation, rather than in the underlying dynamics, is the key to robust, matrix-free solvers for stochastic steady states.

One of the main questions we want to address during further research will be to reduce the stochastic error in the finite-differences approximation of the Jacobian. Variance reduction at the particle level through correlated samples could further stabilize Jacobian-vector products and allow smaller finite-difference steps. In the case of central finite differences, antithetic variates might be a natural variance reduction technique since paired perturbations with opposite randomness can cancel leading-order stochastic fluctuations while preserving the deterministic directional derivative signal.

Beyond computing single steady states, an important next step is to apply Newton--Krylov to compute steady-state branches of parameter-dependent particle systems. Such numerical continuation algorithms require consistent residual evaluations between successive Newton–Krylov steps. Therefore, improving variance reduction at the particle level—through correlated sampling, common-random-number strategies, or smoother estimators will be crucial to enable robust and efficient continuation of stochastic steady states, including the reliable detection of folds and bifurcations in distribution space.

Finally, scaling these approaches to large particle ensembles and to systems with potentially thousands of dimensions remains a major challenge. A key open question is how to identify the most effective sampling strategy for multidimensional CDFs (or their smooth equivalents) constructed via percentile evaluations. What constitutes “best” in this context is not yet defined, but it will likely involve a balance between proximity to the true optimal transport map (much like the Knothe–Rosenblatt rearrangement approximates it) and the computational efficiency of the resulting sampling algorithm. Establishing this balance will be essential for extending smooth timestepper frameworks to high-dimensional stochastic systems like molecular dynamics~\cite{39171} or real economic systems.

\section*{Acknowledgments}
H.V. and I.G.K. acknowledge partial support by the Department of Energy (DOE) under Grant No. DE-SC0024162; I.G.K. also acknowledges partial support by the National Science Foundation under Grants No. CPS2223987 and FDT2436738.  A.C. acknowledges partial support by the National Science Foundation under Grant  CISE CCF 2403452.

\appendix
\section{Derivation of the Relation between the Eigenvalues of $Df$ and $D\phi_h$} \label{app:eigenvalues}
Let 
	\begin{equation} \label{eq:pde}
		\partial_t u_t = f\left(u_t\right)
	\end{equation}
	be a semi-discretized PDE. We are interested in steady-state points $u^*$ such that $f(u^*)=0$. However, in most situations the right-hand side $f$ of the PDE is not available, only a timestepper is. Call $\phi_h(u)$ the flow map of a timestepper with initial condition $u$ over a time interval of size $h$. Steady-states of the PDE~\eqref{eq:pde}, i.e., zeros of $f$, are also zeros of
	\begin{equation*}
		\psi(u) = u - \phi_h(u).
	\end{equation*}
	Additionally, stable steady states of $f$ are stable steady states of $\psi$ and vise versa. In general, the following theorem holds.
	\begin{theorem}
		Let $\lambda_i$ and $\mu_i$ be the respective eigenvalues of $D f(u^*)$ and $D \phi_h(u^*)$. Then $\mu_i = \exp\left(\lambda_i h\right)$. 
	\end{theorem}
	\begin{proof}
		Starting from the integral representation of ODE~\eqref{eq:pde},
		\begin{equation*}
			\phi_h(u) = u + \int_{0}^{h} \partial_s u_s ds = \int_{0}^{h} f\left(u_s\right) ds.
		\end{equation*}
		Then, taking the gradient with respect to the initial condition $u$ ($D=D_u$), we get
		\begin{equation} \label{eq:integral_nabla_phi}
			D \phi_h(u) = I + D_u \int_{0}^{h} f\left(u_s\right) ds = I + \int_{0}^{h} D \left[f\left(u_s\right) \right]ds.
		\end{equation}
		Applying the chain rule $D \left[f\left(u_s\right) \right] = D f\left(u_s\right) D u_s$. However, $u_s$ is just $\phi_s(u)$. Plugging these results into equation~\eqref{eq:integral_nabla_phi}
		\begin{equation*}
			D \phi_h(u) = I + \int_{0}^{h} D f\left(u_s\right) D \phi_s(u) ds.
		\end{equation*}
		For brevity, call $J(s) = D f(u_s)$ and $A(t) = D \phi_t(u)$. We then obtain a compact integral equation
		\begin{equation*}
			A(t) = I + \int_{0}^{h} J(s) A(s) ds,
    	\end{equation*}
    	which is the integral representation of the solution to the matrix ODE
    	\begin{equation*} \label{eq:matrix_ode}
    		\partial_t A(t) = J(t) A(t)
    	\end{equation*}
	   with initial condition $A(0)=I$. This ODE has a unique solution
    	\begin{equation*}
    		A(t) = \exp\left(\int_{0}^t J(s) ds\right).
    	\end{equation*}
	   In steady-state, $J(s)$ is just a constant matrix $J=D f(u^*) $, and $A(t) = \exp\left(t J \right) $. The eigenvalues of $J$ are just $\lambda_i$.  We can conclude that the eigenvalues of $D \phi_h(u^*)=A(h)$ are $\exp\left(h \lambda_i \right)$, and therefore
	\begin{equation*}
		\mu_i = 1- \exp\left(h \lambda_i \right).
	\end{equation*}
\end{proof}

\section{Analytic Steady-State Distribution of the Chemotaxis Model} \label{app:chemotaxis}
The chemotaxis stochastic model~\eqref{eq:chemotaxis} can be seen as a special case of the overdamped Langevin dynamics
\begin{equation*}
    dX_t = -U'(X_t) dt + \sqrt{2 D} dW_t
\end{equation*}
with potential energy 
\begin{equation*}
U(x) = -\int_{-L}^x \chi(S(y)) S_y(y) dy = -\int_{S(-L)}^{S(x)} \chi(S)dS.
\end{equation*}
The invariant distribution of the overdamped Langevin dynamics is
\begin{equation*}
    \mu(x) = Z^{-1} \exp\left(-\frac{1}{D} U(x)\right) = Z^{-1} \exp\left(\frac{1}{D} \int_{-1}^{S(x)} \chi(S)dS\right)
\end{equation*}
It can be seen that the no-flux boundary conditions are automatically satisfied because $J(x) = 0$ everywhere in steady state.
\bibliographystyle{plain}
\bibliography{references}

\end{document}